\documentclass{amsart}
\usepackage[margin=1.5in]{geometry}
\usepackage{amsfonts}

\usepackage{setspace}
\usepackage{graphicx}

\usepackage{hyperref}
\usepackage{graphicx}
\usepackage{amsmath, amsthm, amscd, amsfonts, amssymb, graphicx, color}
 \usepackage{url}
\usepackage{enumerate}
 \makeatletter
\let\reftagform@=\tagform@
\def\tagform@#1{\maketag@@@{(\ignorespaces\textcolor{blue}{#1}\unskip\@@italiccorr)}}
\renewcommand{\eqref}[1]{\textup{\reftagform@{\ref{#1}}}}
\makeatother
\usepackage{hyperref}
\hypersetup{colorlinks=true, linkcolor=red, anchorcolor=green,
citecolor=cyan, urlcolor=red, filecolor=magenta, pdftoolbar=true}

\usepackage{epsfig}        % For postscript
\usepackage{epic,eepic}       % For epic and eepic output from xfig

\newtheorem{theorem}{Theorem}
\newtheorem{rem}{Remark}
\theoremstyle{plain}

\newtheorem{corollary}{Corollary}

\newtheorem{definition}{Definition}

\newtheorem{lemma}{Lemma}
\newtheorem{thm}{Theorem}

\newtheorem{proposition}{Proposition}
\newtheorem{remark}{Remark}

\numberwithin{equation}{section}

\begin{document}

\title[On Pompeiu--Chebyshev functional and its generalization]{On Pompeiu--Chebyshev functional and its generalization}

\author[M.W. Alomari]{Mohammad W. Alomari}

\address{Department of Mathematics, Faculty of Science and
Information Technology, Irbid National University, P.O. Box 2600,
Irbid, P.C. 21110, Jordan.} \email{mwomath@gmail.com}

\date{\today}
\subjclass[2010]{26D15, 26D10, 26A48}

\keywords{Chebyshev functional, Gr\"{u}ss inequality, Pompeiu MVT,
CBS inequality, Hardy inequality.}

\begin{abstract}
In this work, a generalization of Chebyshev functional is
presented. New inequalities of Gr\"{u}ss type via Pompeiu's mean
value theorem are established. Improvements of some old
inequalities are proved. A generalization of pre-Gr\"{u}ss
inequality is elaborated. Some remarks to further generalization
of Chebyshev functional are presented. As applications, bounds for
the reverse of CBS inequality are deduced. Hardy type inequalities
on bounded real interval $\left[a,b\right]$ under some other
circumstances are introduced. Other related ramified inequalities
for differentiable functions are also given.
\end{abstract}

\maketitle

%=============================================================================
\section{Introduction}
%=============================================================================

The difference
\begin{align}
\label{eq1.1.5}\mathcal{T}\left( {f,g} \right) = \frac{1}{{b -
a}}\int_a^b {f\left( t \right)g\left( t \right)dt}  - \frac{1}{{b
- a}}\int_a^b {f\left( t \right)dt}  \cdot \frac{1}{{b -
a}}\int_a^b {g\left( t \right)dt}.
\end{align}
is called `the Chebyshev functional' which it has multiple
applications in several mathematical branches specially in
Numerical integrations and Probability Theory. For more detailed
history see \cite{MPF}.

The most famous bounds of the Chebyshev functional are
incorporated in the following theorem:
\begin{theorem}\label{thm.Chebyshev.bounds}
Let $f,g:[c,d] \to \mathbb{R}$ be two absolutely continuous
functions, then
\begin{align}
&\left|{\mathcal{T}\left( {f,g} \right)} \right|
\nonumber\\
&\le\left\{
\begin{array}{l} \frac{{\left( {d - c} \right)^2 }}{{12}}\left\|
{f^{\prime}} \right\|_\infty  \left\| {g^{\prime}} \right\|_\infty
,\,\,\,\,\,\,\,\,\,{\rm{if}}\,\,f^{\prime},g^{\prime} \in L_{\infty}\left(\left[c,d\right]\right),\,\,\,\,\,\,\,\,{\rm{proved \,\,in \,\,}}{\text{\cite{Cebysev}}}\\
 \\
\frac{1}{4}\left( {M_1 - m_1} \right)\left( {M_2 - m_2}
\right),\,\,\, {\rm{if}}\,\, m_1\le f \le M_1,\,\,\,m_2\le g \le
M_2, \,\,{\rm{proved \,\,in \,\,}}{\text{\cite{Gruss}}}\\
 \\
\frac{{\left( {d - c} \right)}}{{\pi ^2 }}\left\| {f^{\prime}}
\right\|_2 \left\| {g^{\prime}} \right\|_2
,\,\,\,\,\,\,\,\,\,\,\,\,\,\,\,\,{\rm{if}}\,\,f^{\prime},g^{\prime}
\in
L_{2}\left(\left[c,d\right]\right),\,\,\,\,\,\,\,\,\,\,\,{\rm{proved
\,\,in\,\,
}}{\text{\cite{L}}}\\
\\
\frac{1}{8}\left( {d - c} \right)\left( {M - m} \right) \left\|
{h^{\prime}_2} \right\|_{\infty},\,\,\, {\rm{if}}\,\, m\le f \le
M,\,g^{\prime} \in L_{\infty}\left(\left[c,d\right]\right),
\,\,{\rm{proved \,\,in \,\,}}{\text{\cite{O}}}
\end{array} \right. \label{general.gruss.inequality}
\end{align}
The constants $\frac{1}{12}$, $\frac{1}{4}$, $\frac{1}{\pi^2}$ and
$\frac{1}{8}$ are the best possible.
\end{theorem}
The first inequality in \eqref{general.gruss.inequality}, is well
known as Chebyshev inequality (sometimes called Chebyshev second
inequality) which deals with differentiable functions whose first
derivatives are bounded. The second inequality in
\eqref{general.gruss.inequality} is called the Gr\"{u}ss
inequality which deals with integrable bounded functions.

Far away from this,  Pompeiu in \cite{P} established a variant
Mean Value Theorem (MVT) for real functions defined on a real
interval that does not include `$0$'; nowadays known as Pompeiu's
mean value theorem (PMVT), which states that:
\begin{theorem}
 For every real valued function $f$ differentiable on
an interval $[a,b]$ not containing $0$ and for all pairs $x_1\ne
x_2$ in $[a,b]$ there exists a point $\xi\in (x_1,x_2)$ such that
\begin{align}
\frac{{x_1 f\left( {x_2 } \right) - x_2 f\left( {x_1 }
\right)}}{{x_1  - x_2 }} = f\left( \xi \right) - \xi f'\left( \xi
\right).
\end{align}
\end{theorem}
The geometrical interpretation of this theorem as given in
\cite{Sahoo}: the tangent at the point $\left( {\xi ,f\left( \xi
\right)} \right)$ intersects  the $y$-axis at the same point as
the secant line connecting the points $ \left( {x_1 ,f\left( {x_1
} \right)} \right)$ and $ \left( {x_2 ,f\left( {x_2 } \right)}
\right)$. The proof of PMVT can be done by applying the
generalized MVT for derivatives on the functions $x
\mapsto\frac{f\left(x\right)}{x}$ and $x\mapsto \frac{1}{x}$,
where $f$ is differentiable on $[a,b]$. \\

In 2005, and in viewing of PMVT, Pachpatte \cite{Pachpatte}
proposed the following corresponding Chebyshev functional: For
continuous functions $p, q : [a,b] \to \mathbb{R}$ which are
differentiable on $(a, b)$, define
\begin{align}
\label{Pompeiu.Cheb.2}\mathcal{P}  \left( {p,q} \right) = \int_a^b
{p\left( x \right)q\left( x \right)dx}  - \frac{3}{{b^3 - a^3
}}\left( {\int_a^b {xp\left( x \right)dx} } \right)\left(
{\int_a^b {xq\left( x \right)dx} } \right),
\end{align}
and let us call it ``Pompeiu--Chebyshev functional''. In
\cite{Pachpatte}, we find the following result:
\begin{theorem}
\label{thm1.Pachpatte.2005}Let $0\notin [a,b]$. Let $f,g:[a,b] \to
\mathbb{R}$ be two
 continuous functions on $[a,b]$ and differentiable on $(a,b)$.
 Then
\begin{align}
\label{eq2.Pachpatte2005} \left| {\mathcal{P}  \left( {p,q}
\right)} \right| \le \left( {b - a} \right)\left( {1 - \frac{3}{4}
\cdot \frac{{\left( {a + b} \right)^2 }}{{a^2  + ab + b^2 }}}
\right)\left\| {f - \ell f'} \right\|_\infty  \left\| {g - \ell
g'} \right\|_\infty,
\end{align}
where $\ell\left( t \right) = t$, for all $t\in [a,b]$.
\end{theorem}
In the same year, Pe\v{c}ari\'{c} and Ungar \cite{Pecaric} studied
the Pompeiu--Chebyshev functional \eqref{Pompeiu.Cheb.2}, and they
obtained the following result.
\begin{theorem}
Let the functions $f, g: [a,b]\to \mathbb{R}$ be continuous on
$[a,b]$ and differentiable on $(a,b)$ with $0 < a < b$. Then for
$\frac{1}{p}+\frac{1}{q}=1$, with  $1\le p,q\le \infty$, the
inequality
\begin{multline}
\label{eq.pecaric2006}\left| {\mathcal{P } \left( {f,g} \right)}
\right| \le \frac{{\left( {b - a} \right)^{\frac{1}{p}} }}{{\left(
{b^2  - a^2 } \right)}}\left[ {\left\| {f - \ell f'} \right\|_p
\int_a^b {x\left| {g\left( x \right)} \right|A\left( {x,p}
\right)dx} } \right.
\\
\left. { + \left\| {g - \ell g'} \right\|_p \int_a^b {x\left|
{f\left( x \right)} \right|A\left( {x,p} \right)dx} } \right]
\end{multline}
holds, where $\ell (t)=t$, $t\in [a,b]$ and
\begin{multline*}
A\left( {x,p} \right) = \left( {\frac{{a^{2 - q}  - x^{2 - q}
}}{{\left( {1 - 2q} \right)\left( {2 - q} \right)}} + \frac{{x^{2
- q}  - a^{1 + q} x^{1 - 2q} }}{{\left( {1 - 2q} \right)\left( {1
+ q} \right)}}} \right)^{\frac{1}{q}}
\\
+ \left( {\frac{{b^{2 - q}  - x^{2 - q} }}{{\left( {1 - 2q}
\right)\left( {2 - q} \right)}} + \frac{{x^{2 - q}  - b^{1 + q}
x^{1 - 2q} }}{{\left( {1 - 2q} \right)\left( {1 + q} \right)}}}
\right)^{\frac{1}{q}},
\end{multline*}
for all $x\in [a,b]$.
\end{theorem}

\begin{remark}
In particular, we have the following special cases in
\eqref{eq.pecaric2006}:
\begin{align*}
 A\left( {x,\infty } \right) &= \frac{{a^2  + b^2 }}{{2x}} + x - \left( {a + b} \right), \\
 A\left( {x,1} \right) &= \frac{1}{a} + \frac{b}{{x^2 }}, \\
 A\left( {x,2} \right) &= \frac{1}{3}\left[ {\left( {\ln \left( {\frac{x}{a}} \right)^3  + \frac{{a^3 }}{{x^3 }} - 1} \right)^{\frac{1}{2}}  + \left( {\ln \left( {\frac{x}{b}} \right)^3  + \frac{{b^3 }}{{x^3 }} - 1} \right)^{\frac{1}{2}} }
 \right],
\end{align*}
for all $x\in [a,b]$.
\end{remark}

In 2015, Dragomir \cite{Dragomir1} studied the Pompeiu-Chebyshev
functional \eqref{Pompeiu.Cheb.2} with a slightly different
consideration. Namely, Dragomir considered \eqref{Pompeiu.Cheb.2}
as
\begin{align*}
\widehat{\mathcal{P}}\left(f,g\right)=\frac{b^3-a^3}{3}\cdot
\mathcal{P} \left(f,g\right).
\end{align*}

The following lemma plays a main role in the proof of  Dragomir
results in \cite{Dragomir1}.
\begin{lemma}
\label{Dragomir.lemma} Let $f : [a, b]\to \mathbb{C}$ be an
absolutely continuous function on the interval $[a, b]$ with $b >
a > 0$. Then for any $t, x \in [a, b]$, we have
\begin{multline}
\label{eq.Dragomir.lemma} \left| {tf\left( x \right) - xf\left( t \right)} \right| \\
  \le \left\{ \begin{array}{l}
 \left\| {f - \ell f'} \right\|_\infty  \left| {x - t} \right|,\,\,\,\,\,\,\,\,\,\,\,\,\,\,\,\,\,\,\,\,\,\,\,\,\,\,\,\,\,\,\,\,\,\,\,\,\,\,\,\,\,\,\,\,\,\,\,\,\,\,{\rm{if}}\,\,\,\,f - \ell f' \in L^\infty  \left[ {a,b} \right] \\
  \\
 \left( {\frac{1}{{2q - 1}}} \right)^{\frac{1}{q}} \left\| {f - \ell f'} \right\|_p \left| {\frac{{x^q }}{{t^{q - 1} }} - \frac{{t^q }}{{x^{q - 1} }}} \right|^{\frac{1}{q}} ,\,\,\,\,\,\,\,{\rm{if}}\,\,\,\,f - \ell f' \in L^p \left[ {a,b} \right] \\
  \\
 \left\| {f - \ell f'} \right\|_1 \frac{{\max \left\{ {x,t} \right\}}}{{\min \left\{ {x,t} \right\}}},\qquad\,\,\,\,\,\,\,\,\,\,\,\,\,\,\,\,\,\,\,\,\,\,\,\,\,\,\,\,\,\,\,\,\,\,\,\,\,{\rm{if}}\,\,\,\,f - \ell f' \in L^1 \left[ {a,b} \right] \\
 \end{array} \right.
 \end{multline}
where $p,q>1$ with $\frac{1}{p}+\frac{1}{q}=1$.
\end{lemma}
The main results in \cite{Dragomir1}, are combined in the
following theorem.
\begin{thm}
\label{thm1.Dragomir2015}Let $f, g: [a,b]\to \mathbb{C}$ be
absolutely continuous functions on the interval $[a,b]$ with $b >
a > 0$.
\begin{enumerate}
\item If $f^{\prime},g^{\prime}\in L^{\infty}[a,b]$, then
\begin{align}
\label{eq1.Dragomir2015}\left|{\widehat{\mathcal{P}}\left(f,g\right)}\right|
\le \frac{\left(b-a\right)^4}{12} \left\|{f-\ell
f^{\prime}}\right\|_{\infty} \left\|{g-\ell
g^{\prime}}\right\|_{\infty}.
\end{align}
The constant $\frac{1}{12}$ is the best possible.

\item If $f^{\prime}\in L^p$ and $g^{\prime}\in L^{q}[a,b]$,
$p,q>1$, $p\ne 2\ne q$ and $\frac{1}{p}+\frac{1}{q}=1$, then
\begin{align}
\label{eq2.Dragomir2015}\left|{\widehat{\mathcal{P}}\left(f,g\right)}\right|
\le
\frac{M^{\frac{1}{p}}_p\left(a,b\right)M^{\frac{1}{q}}_q\left(a,b\right)}{2\left(2p-1\right)^{\frac{1}{p}}\left(2q-1\right)^{\frac{1}{q}}}
\left\|{f-\ell f^{\prime}}\right\|_{p} \left\|{g-\ell
g^{\prime}}\right\|_{q},
\end{align}
where
\begin{align*}
M_r \left( {a,b} \right): = \int_a^b {\int_a^b {\left| {\frac{{x^r
}}{{t^{r - 1} }} - \frac{{t^r }}{{x^{r - 1} }}} \right|dtdx} }
\end{align*}

\item If $f^{\prime}, g^{\prime}\in L^{2}[a,b]$, then
\begin{align}
\label{eq3.Dragomir2015}\left|{\widehat{\mathcal{P}}\left(f,g\right)}\right|
\le \frac{1}{9} \left\|{f-\ell f^{\prime}}\right\|_{2}
\left\|{g-\ell g^{\prime}}\right\|_{2} \left[ {\left( {a^3  + b^3
} \right)\ln \left(\frac{b}{a} \right)- \frac{2}{3}\left( {b^3  -
a^3 } \right)} \right],
\end{align}
\item If $f^{\prime}, g^{\prime}\in L^{1}[a,b]$, then
\begin{align}
\label{eq4.Dragomir2015}\left|{\widehat{\mathcal{P}}\left(f,g\right)}\right|
\le \frac{\left(b-a\right)^2}{6a} \left(a+2b\right)\left\|{f-\ell
f^{\prime}}\right\|_{1} \left\|{g-\ell g^{\prime}}\right\|_{1},
\end{align}
\end{enumerate}
where $\ell(x)=x$, $x\in [a,b]$.
\end{thm}

\begin{rem}
It is convenient to remark here that, Dragomir proved the
sharpness of \eqref{eq1.Dragomir2015} making use of the constant
functions $f(t)=g(t)=1$; which is trivial case. It would be more
useful if the sharpness holds for non-trivial functions. For this
purpose, we consider $f,g:[a,b]\to \mathbb{R}$ given by
$f(t)=g(t)=ct-1$, $t\in [a,b]$, $b>a>0$ and $c$ any arbitrary
non-zero constant. Simple calculations yield the desired
sharpness.
\end{rem}

Some other inequalities were introduced in literature by many
authors, for recent and related results we refer the reader to
\cite{Acu1}--\cite{Acu3}, \cite{Dragomir}--\cite{Dragomir2},
\cite{Matic} and \cite{Popa}.

This work is divided into  five sections, after this introduction,
the second and third  sections are devoted to elaborate and
investigate some new inequalities of Gr\"{u}ss type via Pompeiu's
mean value theorem. Improvements of some old inequalities are also
provided. In section 4, generalizations of Gr\"{u}ss type
inequalities via Boggio mean value theorem are established. As
applications, bounds for the reverse CBS inequality are obtained.
In section 5, Using some extracted functionals; new Hardy type
inequalities and their generalizations are detected. Some other
inequalities for differentiable functions are also given.

%=============================================================================
\section{The Results}\label{sec2}
%=============================================================================
Let us start with the following results regarding positivity of
$\widehat{\mathcal{P}}\left(f,g\right)$.
\begin{thm}
\label{thm1.1.1}Let $a,b\in\mathbb{R}$ with $b>a>0$. Let $f,g:
\left[ {a,b} \right]\subseteq \mathbb{R}\to \mathbb{R}$ be two
Lebesgue integrable functions on $\left[ {a,b} \right]$ and
satisfy the condition
\begin{align*}
\left( {yf\left( x \right) - xf\left( y \right)} \right)\left(
{xg\left( y \right) - yg\left( x \right)} \right) \le  0,
\end{align*}
for all $x,y\in [a,b]$. Then
\begin{align}
\label{eq3.1.Alomari2017} \frac{b^3 - a^3 }{3}\int_a^b {f\left( x
\right)g\left( x \right)dx} \ge \int_a^b {x f\left( x \right)dx}
\int_a^b {x g\left( x \right)dx} .
\end{align}
\end{thm}

\begin{proof}
Let $f$ and $g$ be satisfied  the given condition
\begin{align*}
\left( {yf\left( x \right) - xf\left( y \right)} \right)\left(
{xg\left( y \right) - yg\left( x \right)} \right) \le  0,
\end{align*}
for all $x, y\in [a,b]$, therefore we have
\begin{align*}
\int_a^b {\int_a^b {\left( {yf\left( x \right) - xf\left( y
\right)} \right)\left( {xg\left( y \right) - yg\left( x \right)}
\right)dxdy} }\le 0.
\end{align*}
but also, we have
\begin{align*}
&\frac{1}{2}\int_a^b {\int_a^b {\left( {yf\left( x \right) -
xf\left( y \right)} \right)\left( {xg\left( y \right) - yg\left( x
\right)} \right)dxdy} }
\\
&= \frac{1}{2}\int_a^b {\int_a^b {\left\{ {xyf\left( x
\right)g\left( y \right) - y^2f\left( x \right)g\left( x \right)
-x^2 f\left( y \right)g\left( y \right) + xyf\left( y
\right)g\left( x \right)} \right\}dxdy} }
\\
&= \left( {\int_a^b {xf\left( x \right)dx} } \right)\left(
{\int_a^b {yg\left( y \right)dy} } \right) - \frac{{b^3  - a^3
}}{3}\int_a^b {f\left( x \right)g\left( x \right)dx}
\\
&\le 0,
\end{align*}
which proves the inequality \eqref{eq3.1.Alomari2017}.
\end{proof}

\begin{corollary}
\label{cor1.1.1} Let $a,b\in\mathbb{R}$ with $b>a>0$. Let $f,g:
\left[ {a,b} \right]\subseteq \mathbb{R}\to \mathbb{R}$ be two
Lebesgue integrable functions on $\left[ {a,b} \right]$ and
satisfy the condition
\begin{align*}
\left( {yf\left( x \right) - xf\left( y \right)} \right)\left(
{xg\left( y \right) - yg\left( x \right)} \right) \ge  0,
\end{align*}
for all $x,y\in [a,b]$. Then
\begin{align}
\label{eq3.2.Alomari2017} \frac{b^3 - a^3 }{3}\int_a^b {f\left( x
\right)g\left( x \right)dx} \le \int_a^b {x f\left( x \right)dx}
\int_a^b {x g\left( x \right)dx} .
\end{align}
\end{corollary}

\begin{remark}
Let $a,b\in\mathbb{R}$ with $b>a>0$. Let
$h:\left[a,b\right]\to\mathbb{R}$ be  an increasing on
$\left[a,b\right]$ then the function
$f:\left[a,b\right]\to\mathbb{R}$ given by
$f\left(t\right)=\frac{h\left( t \right)}{t}$ is increasing. So
that for any two distinct points $x,y\in [a,b]$ with $x\ge y$, we
have $yf\left(x\right)-xf\left(y\right)=h\left( x \right)-h\left(
y \right)\ge 0$. The reverse observation holds for
 decreasing function $h$. A generalization of monotonicity and thus
 the previous two results are given in Section \ref{sec4}.
\end{remark}

A pre-Gr\"{u}ss like inequality is incorporated in the following
theorem:
\begin{thm}
\label{thm1.Alomari2017}Let $a,b\in\mathbb{R}$ with $b>a>0$. Let
$f,g:\left[a,b\right]\to\mathbb{R}$ be two Lebesgue integrable
functions then the inequality
\begin{align}
\label{eq1.Alomari2017}\left|{\widehat{\mathcal{P}}\left(f,g\right)}\right|\le
\left|{\widehat{\mathcal{P}}\left(f,f\right)}\right|^{\frac{1}{2}}\cdot
\left|{\widehat{\mathcal{P}}\left(g,g\right)}\right|^{\frac{1}{2}},
\end{align}
holds and sharp.
\end{thm}

\begin{proof}
It is easy to verify that
\begin{align}
\label{eq2.Alomari2017}\widehat{\mathcal{P}}\left(f,g\right)=\frac{1}{2}\int_a^b
{\int_a^b {\left( {tf\left( x \right) - xf\left( t \right)}
\right)\left( {tg\left( x \right) - xg\left( t \right)}
\right)dtdx} }.
\end{align}
For instance, we observe that
\begin{align*}
 &\frac{1}{2}\int_a^b {\int_a^b {\left( {tf\left( x \right) - xf\left( t \right)} \right)^2 dtdx} }  \\
  &= \frac{1}{2}\left[ {\int_a^b {\int_a^b {t^2  f^2\left( x \right)  dtdx} }  + \int_a^b {\int_a^b {x^2  f^2\left( t \right)  dtdx} } } \right. \\
 &\qquad\left. { - 2\int_a^b {t f\left( t \right) dt} \int_a^b {x f\left( x \right) dx} } \right] \\
  &= \frac{1}{2}\left[ {\frac{{b^3  - a^3 }}{3} \cdot 2\int_a^b { f^2\left( x \right)  dx}  - 2\int_a^b {t f\left( t \right)dt} \int_a^b {x f\left( x \right) dx} } \right] \\
  &= \frac{{b^3  - a^3 }}{3}\int_a^b { f^2\left( x \right)dx}  -  \left( \int_a^b {xf\left( x \right)dx} \right)^2  \\
  &=  \widehat{\mathcal{P}}\left(f,f\right).
\end{align*}
Now, using the triangle integral inequality and the Cauchy-Schwarz
inequality, we have
\begin{align*}
&\left|{\widehat{\mathcal{P}}\left(f,g\right)}\right|
\\
&=\frac{1}{2}\left| {\int_a^b {\int_a^b {\left( {tf\left( x
\right) - xf\left( t \right)} \right)\left( {tg\left( x \right) -
xg\left( t \right)} \right)dtdx} } } \right|
\\
&\le \left(\frac{1}{2} {\int_a^b {\int_a^b {\left| {tf\left( x
\right) - xf\left( t \right)} \right|^2 dtdx} } }
\right)^{\frac{1}{2}} \left( {\frac{1}{2}\int_a^b {\int_a^b
{\left| {tg\left( x \right) - xg\left( t \right)} \right|^2 dtdx}
} } \right)^{\frac{1}{2}}
\\
&=\left|{\widehat{\mathcal{P}}\left(f,f\right)}\right|^{\frac{1}{2}}\cdot
\left|{\widehat{\mathcal{P}}\left(g,g\right)}\right|^{\frac{1}{2}},
\end{align*}
as desired. The sharpness follows by considering $f(t)=g(t)=ct-1$,
$t\in [a,b]$, $b>a>0$ and $c$ any arbitrary non-zero constant.
\end{proof}

\begin{thm}
\label{thm2.Alomari2017}Let $a,b\in\mathbb{R}$ with $b>a>0$. Let
$f,g:\left[a,b\right]\to\mathbb{R}$ be two measurable functions.
If there exists real numbers $\Phi, \phi, \gamma, \Gamma $ such
that $\phi \le f\left(s\right)\le \Phi$ and $\gamma \le
g\left(s\right)\le \Gamma$ for all $s\in \left[a,b\right]$, then
the inequality
\begin{align}
\label{eq3.Alomari2017}\left|{\widehat{\mathcal{P}}\left(f,g\right)}\right|\le
\frac{1}{2}\left( {b - a} \right)^2\left( {b\Phi  - a\phi }
\right) \left( {b\Gamma  - a\gamma } \right)
\end{align}
holds.
\end{thm}

\begin{proof}
Since $\gamma \le g\left(s\right)\le \Gamma$ for all $s\in
\left[a,b\right]$, then $a\gamma  \le tg\left( x \right) \le
b\Gamma$ and $- b\Gamma  \le  - xg\left( t \right) \le  - a\gamma$
adding the last two inequalities we get that $a\gamma  - b\Gamma
\le tg\left( x \right) - xg\left( t \right) \le b\Gamma  - a\gamma
$ or we may write  $\left|{tg\left( x \right) - xg\left( t
\right)}\right| \le b\Gamma  - a\gamma $. Similarly, for  $f$ we
have  $\left|{tf\left( x \right) - xf\left( t \right)}\right| \le
b\Phi  - a\phi $. So that we have
\begin{align*}
\left|{\widehat{\mathcal{P}}\left(f,g\right)}\right|
&=\frac{1}{2}\left| {\int_a^b {\int_a^b {\left( {tf\left( x
\right) - xf\left( t \right)} \right)\left( {tg\left( x \right) -
xg\left( t \right)} \right)dtdx} } } \right|
\\
&\le \frac{1}{2} \int_a^b {\int_a^b {\left| {tf\left( x \right) -
xf\left( t \right)} \right|\left| {tg\left( x \right) - xg\left( t
\right)} \right|dt} dx}
\\
&\le\frac{1}{2}\left( {b - a} \right)^2\left( {b\Phi  - a\phi }
\right) \left( {b\Gamma  - a\gamma } \right),
\end{align*}
which proves the result \eqref{eq3.Alomari2017}.
\end{proof}

\begin{thm}
\label{thm3.Alomari2017}Let $a,b\in\mathbb{R}$ with $b>a>0$. Let
$f:\left[a,b\right]\to\mathbb{R}$ be a Lebesgue  integrable and
$g:\left[a,b\right]\to\mathbb{R}$ be a measurable and  there
exists real numbers $\gamma,\Gamma $ such that $\gamma \le
g\left(s\right)\le \Gamma$ for all $s\in \left[a,b\right]$, then
the inequality
\begin{align}
\label{eq4.Alomari2017}\left|{\widehat{\mathcal{P}}\left(f,g\right)}\right|\le
\frac{1}{\sqrt{2}}\left( {b - a} \right)\left( {b\Gamma  - a\gamma
} \right)
\cdot\left|{\widehat{\mathcal{P}}\left(f,f\right)}\right|^{\frac{1}{2}}
\end{align}
holds.
\end{thm}

\begin{proof}
Since $\gamma \le g\left(s\right)\le \Gamma$ for all $s\in
\left[a,b\right]$, then $a\gamma  \le tg\left( x \right) \le
b\Gamma$ and $- b\Gamma  \le  - xg\left( t \right) \le  - a\gamma$
adding the last two inequalities we get that $a\gamma  - b\Gamma
\le tg\left( x \right) - xg\left( t \right) \le b\Gamma  - a\gamma
$ or we may write  $\left|{tg\left( x \right) - xg\left( t
\right)}\right| \le b\Gamma  - a\gamma $.

Now since
\begin{align*}
\widehat{\mathcal{P}}\left(g,g\right) =\frac{1}{2}\int_a^b
{\int_a^b {\left| {tg\left( x \right) - xg\left( t \right)}
\right|^2 dtdx} },
\end{align*}
it follows that $ \widehat{\mathcal{P}}\left(g,g\right)  \le
\frac{1}{2}\left( {b\Gamma  - a\gamma } \right)^2 \left( {b - a}
\right)^2$. Employing \eqref{eq1.Alomari2017} we get the required
result.
\end{proof}
An improvement of \eqref{eq3.Alomari2017} is given in the
following theorem.
\begin{thm}
\label{thm4.Alomari2017}Let $a,b\in\mathbb{R}$ with $b>a>0$. Let
$f,g:\left[a,b\right]\to\mathbb{R}$ be two measurable functions on
$[a,b]$. If there exists real numbers $\Phi, \phi, \gamma, \Gamma
$ such that $\phi \le f\left(s\right)\le \Phi$ and $\gamma \le
g\left(s\right)\le \Gamma$ for all $s\in \left[a,b\right]$, that
satisfies the condition $\phi \Gamma -\Phi \gamma \ne0$, then the
inequality
\begin{multline}
\label{eq5.Alomari2017}\left|{\widehat{\mathcal{P}}\left(f,g\right)}\right|
\\
\le   \left\{ \begin{array}{l}
 \frac{1}{8}\left( {b^2  - a^2 } \right)^2 \left( {\Gamma ^2  - \gamma ^2 } \right)\left( {\Phi ^2  - \phi ^2 } \right),\qquad\qquad\qquad\qquad\qquad\,\,\,\,\,{\rm{If}}\,\,\,\Gamma a - \gamma b > 0,\Phi a - \phi b > 0 \\
 \frac{1}{8}\left( {b^2  - a^2 } \right)\left( {\Gamma ^2  - \gamma ^2 } \right) \cdot \left[ {\left( {\Phi b - \phi a} \right)^2  + \left( {\Phi a - \phi b} \right)^2 } \right],\,\,\,\,\,\qquad\,\,\,\,\,{\rm{If}}\,\,\,\Gamma a - \gamma b > 0,\Phi a - \phi b < 0 \\
 \frac{1}{8}\left( {b^2  - a^2 } \right)\left( {\Phi ^2  - \phi ^2 } \right)\left[ {\left( {\Gamma b - \gamma a} \right)^2  + \left( {\Gamma a - \gamma b} \right)^2 } \right],\,\,\,\,\,\qquad\,\,\,\,\,\,\,\,\,{\rm{If}}\,\,\,\Gamma a - \gamma b < 0,\Phi a - \phi b > 0 \\
 \frac{1}{8}\left[ {\left( {\Gamma b - \gamma a} \right)^2  + \left( {\Gamma a - \gamma b} \right)^2 } \right] \cdot \left[ {\left( {\Phi b - \phi a} \right)^2  + \left( {\Phi a - \phi b} \right)^2 } \right],\,\,\,\,{\rm{If}}\,\,\,\Gamma a - \gamma b < 0,\Phi a - \phi b < 0 \\
 \end{array} \right.
\end{multline}
holds.
\end{thm}

\begin{proof}
Since $\gamma \le g\left(s\right)\le \Gamma$ for all $s\in
\left[a,b\right]$, then $tg\left( x \right) \le \Gamma t$, and $-
xg\left( t \right) \le  - \gamma x$, adding these two inequalities
we get that $tg\left( x \right) - xg\left( t \right) \le \Gamma t
- \gamma x$. Similarly, for  $f$ we have $ tf\left( x \right) -
xf\left( t \right)  \le  \Phi t - \phi x$. So that we have
\begin{align}
\left|{\widehat{\mathcal{P}}\left(f,g\right)}\right|
&=\frac{1}{2}\left| {\int_a^b {\int_a^b {\left( {tf\left( x
\right) - xf\left( t \right)} \right)\left( {tg\left( x \right) -
xg\left( t \right)} \right)dtdx} } } \right|
\nonumber\\
&\le \frac{1}{2} \int_a^b {\int_a^b {\left| {tf\left( x \right) -
xf\left( t \right)} \right|\left| {tg\left( x \right) - xg\left( t
\right)} \right|dt} dx}
\nonumber\\
&\le\frac{1}{2} \int_a^b {\int_a^b {\left| {\Phi t - \phi x}
\right| \left| {\Gamma t - \gamma x} \right|dt}dx}. \label{eq3.6}
\end{align}
Now, substituting $z=\Gamma t - \gamma x$ and $y=\Phi t - \phi x$,
where $\Gamma a - \gamma b\le z\le\Gamma b - \gamma a$ and $\Phi a
- \phi b\le y\le\Phi b - \phi a$. Solving the last two equation
with respect to $x$ and $t$ we find that
\begin{align*}
x=\frac{\Phi z - \Gamma y}{\phi \Gamma -\Phi \gamma},\qquad
\text{and}\qquad t=\frac{\phi z - \gamma y}{\phi \Gamma -\Phi
\gamma}.
\end{align*}
Clearly, the Jacobian $J\left(z,y\right) =1$, and thus we have
\begin{align}
\int_a^b {\int_a^b {\left| {\Phi t - \phi x} \right| \left|
{\Gamma t - \gamma x} \right|dt}dx} &=\int_{\Phi a - \phi b}^{\Phi
b - \phi a} {\int_{\Gamma a - \gamma b}^{\Gamma b - \gamma a}
{\left| {y} \right| \left| {z} \right|J\left(z,y\right)dz}dy}.
\label{eq3.7}
\end{align}
To evaluate the integral in \eqref{eq3.7}, we have the following
cases:\\

\textbf{Case I:} If $\Phi a - \phi b>0$ and $\Gamma a - \gamma
b>0$, then
\begin{align}
\int_{\Phi a - \phi b}^{\Phi b - \phi a} {\int_{\Gamma a - \gamma
b}^{\Gamma b - \gamma a} {\left| {y} \right| \left| {z}
\right|dz}dy}&=  \frac{{\left( {\Gamma b - \gamma a} \right)^2 -
\left( {\Gamma a - \gamma b} \right)^2 }}{2} \cdot \frac{{\left(
{\Phi b - \phi a} \right)^2  - \left( {\Phi a - \phi b} \right)^2
}}{2}
\nonumber\\
&= \frac{1}{4}\left( {b^2  - a^2 } \right)^2 \left( {\Gamma ^2 -
\gamma ^2 } \right)\left( {\Phi ^2  - \phi ^2 } \right),
\label{eq3.8}
\end{align}
substituting \eqref{eq3.8} in \eqref{eq3.6} we get the first
inequality in \eqref{eq5.Alomari2017}.\\

\textbf{Case II:} If $\Phi a - \phi b<0$ and $\Gamma a - \gamma
b>0$, then
\begin{align}
\int_{\Phi a - \phi b}^{\Phi b - \phi a} {\int_{\Gamma a - \gamma
b}^{\Gamma b - \gamma a} {\left| {y} \right| \left| {z}
\right|dz}dy}&=   \frac{{\left( {\Gamma b - \gamma a} \right)^2 -
\left( {\Gamma a - \gamma b} \right)^2 }}{2} \cdot \frac{{\left(
{\Phi b - \phi a} \right)^2  + \left( {\Phi a - \phi b} \right)^2
}}{2}
\nonumber\\
&= \frac{1}{4}\left( {b^2  - a^2 } \right)\left( {\Gamma ^2  -
\gamma ^2 } \right) \cdot \left[ {\left( {\Phi b - \phi a}
\right)^2  + \left( {\Phi a - \phi b} \right)^2 } \right],
\label{eq3.9}
\end{align}
substituting \eqref{eq3.9} in \eqref{eq3.6} we get the fourth
inequality in \eqref{eq5.Alomari2017}.\\

\textbf{Case III:} If $\Phi a - \phi b>0$ and $\Gamma a - \gamma
b<0$, then
\begin{align}
\int_{\Phi a - \phi b}^{\Phi b - \phi a} {\int_{\Gamma a - \gamma
b}^{\Gamma b - \gamma a} {\left| {y} \right| \left| {z}
\right|dz}dy}&=   \frac{{\left( {\Gamma b - \gamma a} \right)^2 +
\left( {\Gamma a - \gamma b} \right)^2 }}{2} \cdot \frac{{\left(
{\Phi b - \phi a} \right)^2  - \left( {\Phi a - \phi b} \right)^2
}}{2}
\nonumber\\
&=\frac{1}{4}\left( {b^2  - a^2 } \right)\left( {\Phi ^2  - \phi
^2 } \right)\left[ {\left( {\Gamma b - \gamma a} \right)^2  +
\left( {\Gamma a - \gamma b} \right)^2 } \right],
 \label{eq3.10}
\end{align}
substituting \eqref{eq3.10} in \eqref{eq3.6} we get the third
inequality in \eqref{eq5.Alomari2017}.\\

\textbf{Case IV:} If $\Phi a - \phi b<0$ and $\Gamma a - \gamma
b<0$, then
\begin{align}
\int_{\Phi a - \phi b}^{\Phi b - \phi a} {\int_{\Gamma a - \gamma
b}^{\Gamma b - \gamma a} {\left| {y} \right| \left| {z}
\right|dz}dy}=   \frac{{\left( {\Gamma b - \gamma a} \right)^2 +
\left( {\Gamma a - \gamma b} \right)^2 }}{2} \cdot \frac{{\left(
{\Phi b - \phi a} \right)^2  + \left( {\Phi a - \phi b} \right)^2
}}{2}, \label{eq3.11}
\end{align}
substituting \eqref{eq3.11} in \eqref{eq3.6} we get the second
inequality in \eqref{eq5.Alomari2017}.\\
\end{proof}

\begin{remark}
Theorem \ref{thm4.Alomari2017} does not work for  two identical
functions. In other words, we cannot choose $f=g$ or $f=g$ a.e. on
$[a,b]$. The reason behind this is that, if one chooses $f=g$ then
the substitution $z=\Gamma t - \gamma x$ and $y=\Phi t - \phi x$,
will be the same and so that we have $\phi \Gamma -\Phi \gamma
=\phi \Phi -\Phi \phi =0$ and this contradicts the assumption that
$\phi \Gamma -\Phi \gamma \ne0$.
\end{remark}

More general and extensive case can be done through the following
improvement of \eqref{eq4.Alomari2017}.
\begin{thm}
\label{thm5.Alomari2017}Under the assumptions of Theorem
\ref{thm3.Alomari2017}, then the inequality
\begin{align}
\label{eq3.12.Alomari2017}\left|{\widehat{\mathcal{P}}\left(f,g\right)}\right|\le
\frac{1}{2\sqrt{3}}\left|{\widehat{\mathcal{P}}\left(f,f\right)}\right|^{\frac{1}{2}}
\left[ {2\left( {b^3  - a^3 } \right)\left( {b - a} \right)\left(
{\Gamma ^2  + \gamma ^2 } \right) - 3\Gamma \gamma \left( {b^2  -
a^2 } \right)^2 } \right]^{\frac{1}{2}}
\end{align}
holds.
\end{thm}

\begin{proof}
From Theorem  \ref{thm1.Alomari2017}, we can state the following
\begin{align*}
&\left|{\widehat{\mathcal{P}}\left(f,g\right)}\right|
\\
&=\frac{1}{2}\left| {\int_a^b {\int_a^b {\left( {tf\left( x
\right) - xf\left( t \right)} \right)\left( {tg\left( x \right) -
xg\left( t \right)} \right)dtdx} } } \right|
\\
&\le \left(\frac{1}{2} {\int_a^b {\int_a^b {\left| {tf\left( x
\right) - xf\left( t \right)} \right|^2 dtdx} } }
\right)^{\frac{1}{2}} \left( {\frac{1}{2}\int_a^b {\int_a^b
{\left| {tg\left( x \right) - xg\left( t \right)} \right|^2 dtdx}
} } \right)^{\frac{1}{2}}
\\
&\le
\left|{\widehat{\mathcal{P}}\left(f,f\right)}\right|^{\frac{1}{2}}\left(
{\frac{1}{2}\int_a^b {\int_a^b  {\left| {\Gamma t - \gamma x}
\right|^2  dtdx} } } \right)^{\frac{1}{2}}
\\
&\le
\left|{\widehat{\mathcal{P}}\left(f,f\right)}\right|^{\frac{1}{2}}
\left( \frac{1}{{12}}\left[ {2\left( {b^3  - a^3 } \right)\left(
{b - a} \right)\left( {\Gamma ^2  + \gamma ^2 } \right) - 3\Gamma
\gamma \left( {b^2  - a^2 } \right)^2 } \right]
\right)^{\frac{1}{2}},
\end{align*}
which gives the desired result in \eqref{eq3.12.Alomari2017}.
\end{proof}
\begin{remark}
In Theorem \ref{thm5.Alomari2017}, if $g=f$ then
\begin{align*}
\left|{\widehat{\mathcal{P}}\left(g,g\right)}\right|\le
\frac{1}{12} \left[ {2\left( {b^3  - a^3 } \right)\left( {b - a}
\right)\left( {\Gamma ^2  + \gamma ^2 } \right) - 3\Gamma \gamma
\left( {b^2  - a^2 } \right)^2 } \right].
\end{align*}
So that, if there exist $\phi \le f(x) \le \Phi$, for all $x\in
[a,b]$, then
\begin{align*}
\left|{\widehat{\mathcal{P}}\left(f,g\right)}\right| &\le
\left|{\widehat{\mathcal{P}}\left(f,f\right)}\right|^{1/2}\left|{\widehat{\mathcal{P}}\left(g,g\right)}\right|^{1/2}
\\
&\le \frac{1}{12} \left[ {2\left( {b^3  - a^3 } \right)\left( {b -
a} \right)\left( {\Gamma ^2  + \gamma ^2 } \right) - 3\Gamma
\gamma \left( {b^2  - a^2 } \right)^2 } \right]^{1/2}
\\
&\qquad\times \left[ {2\left( {b^3  - a^3 } \right)\left( {b - a}
\right)\left( {\Phi ^2  + \phi ^2 } \right) - 3\Phi \phi \left(
{b^2  - a^2 } \right)^2 } \right]^{1/2},
\end{align*}
which improves and generalizes \eqref{eq5.Alomari2017}.
\end{remark}
Next, an improvement of \eqref{eq3.Dragomir2015} can be presented
as follows:
\begin{thm}
\label{thm6.Alomari2017}Let $a,b\in\mathbb{R}$ with $b>a>0$. Let
$f,g:\left[a,b\right]\to\mathbb{R}$ be two absolutely continuous
functions. If $f^{\prime}, g^{\prime}\in L_2[a,b]$, then we have
\begin{align}
\label{eq6.Alomari2017}
\left|{\widehat{\mathcal{P}}\left(f,g\right)}\right|
\le\frac{2}{{9\pi ^2 }} \frac{b^2}{a^6 }\left( {b^3 - a^3 }
\right)^2\left( {b  - a} \right) \cdot \left\| {\ell f^{\prime} -
f} \right\|_2\left\| {\ell g^{\prime} - g} \right\|_2,
\end{align}
where $\ell\left(t\right)=t$, $t\in \left[a,b\right]$.
\end{thm}

\begin{proof}
Since we have
\begin{multline}
\label{eq7.Alomari2017}\left|{\widehat{\mathcal{P}}\left(f,g\right)}\right|
\\\le\left( {\frac{1}{2}\int_a^b {\int_a^b {\left| {tf\left( x
\right) - xf\left( t \right)} \right|^2 dtdx} } }
\right)^{\frac{1}{2}} \left( {\frac{1}{2}\int_a^b {\int_a^b
{\left| {tg\left( x \right) - xg\left( t \right)} \right|^2 dtdx}
} } \right)^{\frac{1}{2}}.
\end{multline}
Let us write
\begin{align}
\int_a^b {\left( {tf\left( x \right) - xf\left( t \right)}
\right)^2  dx} &=t^2\int_a^b {x^2 \left( {\frac{{f\left( x
\right)}}{x} - \frac{{f\left( t \right)}}{t}} \right)^2
dx}.\label{eq8.Alomari2017}
\end{align}
In \cite{Milovanovic}, G. Milovanovi\'{c} and \v{Z}.
Milovanovi\'{c} proved the following inequality:
\begin{multline}
\label{Milo}\int_a^b {p\left( s \right)\left( {F\left( s \right) -
F\left( \eta  \right)} \right)^2 ds}
\\
\le \frac{4}{{\pi ^2 }}\left[ {\max \left\{ {\int_a^\eta  {p\left(
s \right)ds} ,\int_\eta ^b {p\left( s \right)ds} } \right\}}
\right]^2\int_a^b {r\left( s \right)\left( {F'\left( s \right)}
\right)^2 ds}
\end{multline}
for any $\eta \in \left[a,b\right]$. where $p$ ia positive and
continuous on $[a,b]$ with $\int_a^b{p\left(s\right)ds} < \infty$
and $F$ is absolutely continuous on $(a,b)$ with $ \int_a^b
{r\left( s \right)\left( {F^{\prime}\left( s \right)} \right)^2
ds} <\infty$ and $r\left( s \right)=\frac{1}{p\left(s\right)}$.
The inequality is sharp.

In viewing of \eqref{eq8.Alomari2017} and by setting where
$F\left(s\right)= \frac{{f\left( s \right)}}{s}$ and
$p\left(s\right)=s^2$, with $\eta=t$ for all $s,t\in [a,b]$, with
$b>a>0$
in \eqref{Milo} we can state that% \label{eq7.Alomari2017}
\begin{align*}
&\int_a^b {\left( {tf\left( x \right) - xf\left( t \right)}
\right)^2  dx}
\nonumber\\
&=t^2\int_a^b {x^2 \left( {\frac{{f\left( x \right)}}{x} -
\frac{{f\left( t \right)}}{t}} \right)^2 dx}
\nonumber\\
&\le t^2\cdot \frac{4}{{\pi ^2 }}\left[ {\max \left\{ {\int_a^t
{s^2ds} ,\int_t^b {s^2ds} } \right\}} \right]^2\int_a^b
{\frac{1}{s^2}\left( {\left({\frac{{f\left( s \right)}}{s}
}\right)^{\prime}} \right)^2 ds}
\nonumber\\
&=t^2\cdot \frac{4}{{\pi ^2 }} \left[ {\max \left\{ {\frac{{t ^3 -
a^3 }}{3},\frac{{b^3  - t ^3 }}{3}} \right\}} \right]^2 \int_a^b
{\frac{1}{s^2}\left( { \frac{{s f^{\prime}\left( s \right)-f
\left( s \right)}}{s^2} } \right)^2 ds}
\nonumber\\
&=t^2\cdot \frac{4}{{\pi ^2 }} \left[ {\frac{{b^3  - a^3 }}{6} +
\frac{1}{3}\left| {t^3  - \frac{{a^3  + b^3 }}{2}} \right|}
\right]^2\cdot\mathop {\max }\limits_{s \in \left[ {a,b} \right]}
\left\{ {\frac{1}{{s^6 }}} \right\} \cdot\int_a^b { \left( {s
f^{\prime}\left( s \right)-f \left( s \right)} \right)^2 ds}
\nonumber\\
&=t^2\cdot \frac{4}{{\pi ^2 }} \frac{1}{a^6 }\left[ {\frac{{b^3 -
a^3 }}{6} + \frac{1}{3}\left| {t^3  - \frac{{a^3  + b^3 }}{2}}
\right|} \right]^2\cdot  \left\| {\ell f^{\prime} - f}
\right\|_2^2.
\end{align*}
Integrating w. r. to $t$ over $[a,b]$, we get
\begin{align}
&\int_a^b {\int_a^b {\left( {tf\left( x \right) - xf\left( t
\right)} \right)^2  dx}dt}
\nonumber\\
&\le \frac{4}{{\pi ^2 }} \frac{1}{a^6 }  \left\| {\ell f^{\prime}
- f} \right\|_2^2\cdot\int_a^b {t^2\cdot \left[ {\frac{{b^3 - a^3
}}{6} + \frac{1}{3}\left| {t^3  - \frac{{a^3  + b^3 }}{2}}
\right|} \right]^2dt}\label{eq10.Alomari2017}
 \\
&\le \frac{4}{{\pi ^2 }} \frac{1}{a^6 }  \left\| {\ell f^{\prime}
- f} \right\|_2^2\cdot\left( {b  - a} \right)\mathop {\sup
}\limits_{t \in \left[ {a,b} \right]}{t^2\cdot \left[ {\frac{{b^3
- a^3 }}{6} + \frac{1}{3}\left| {t^3  - \frac{{a^3  + b^3 }}{2}}
\right|} \right]^2 }
\nonumber\\
&\le \frac{4}{{\pi ^2 }} \frac{1}{a^6 }  \left\| {\ell f^{\prime}
- f} \right\|_2^2\cdot\left( {b  - a} \right) b^2 \frac{{\left(
{b^3  - a^3 } \right)^2 }}{9}.\label{eq11.Alomari2017}
\end{align}
Similarly, for $g$ we have
\begin{align}
\int_a^b {\left( {tg\left( x \right) - xg\left( t \right)}
\right)^2  dx}\le \frac{4}{{\pi ^2 }} \frac{1}{a^6 }  \left\|
{\ell g^{\prime} - g} \right\|_2^2\cdot\left( {b  - a} \right) b^2
\frac{{\left( {b^3  - a^3 } \right)^2
}}{9}.\label{eq12.Alomari2017}
\end{align}
Substituting \eqref{eq11.Alomari2017} and \eqref{eq12.Alomari2017}
in \eqref{eq7.Alomari2017} we get the desired result
\eqref{eq6.Alomari2017}.
\end{proof}
\begin{corollary}
\label{cor1}Let $a,b\in\mathbb{R}$ with $b>a>0$. Let
$f:\left[a,b\right]\to\mathbb{R}$ be an absolutely continuous
function such that  $f^{\prime}\in L_2[a,b]$ and
$g:\left[a,b\right]\to\mathbb{R}$ is Lebesgue integrable on
$[a,b]$, then we have
\begin{align}
\label{eq16.Alomari2017}
\left|{\widehat{\mathcal{P}}\left(f,g\right)}\right|
\le\frac{\sqrt{2}}{{3\pi }} \frac{b}{a^3 }\left( {b^3 - a^3 }
\right) \left( {b  - a} \right)^{1/2}\cdot \left\| {\ell
f^{\prime} - f}
\right\|_2\left|{\widehat{\mathcal{P}}\left(g,g\right)}\right|^{\frac{1}{2}},
\end{align}
where $\ell\left(t\right)=t$, $t\in \left[a,b\right]$.
\end{corollary}
\begin{proof}
The result follows from \eqref{eq1.Alomari2017} and
\eqref{eq6.Alomari2017}.
\end{proof}

\begin{corollary}
\label{cor2}Let $a,b\in\mathbb{R}$ with $b>a>0$. Let
$a,b\in\mathbb{R}$ with $b>a>0$. Let
$f:\left[a,b\right]\to\mathbb{R}$ be an absolutely continuous
function such that $f^{\prime}\in L_2[a,b]$ and
$g:\left[a,b\right]\to\mathbb{R}$ be a measurable function  such
that $\gamma \le g\left(s\right)\le \Gamma$ for all $s\in
\left[a,b\right]$ and for some real numbers $\gamma,\Gamma$, then
the inequality
\begin{align}
\label{eq16.Alomari2017}\left|{\widehat{\mathcal{P}}\left(f,g\right)}\right|\le
\frac{b}{{3\pi a^3}} \left( {b - a} \right)\left( {b\Gamma  -
a\gamma } \right)  \left( {b^3 - a^3 } \right) \cdot \left\| {\ell
f^{\prime} - f} \right\|_2
\end{align}
holds.
\end{corollary}

\begin{proof}
Substituting \eqref{eq3.Alomari2017} and \eqref{eq6.Alomari2017}
in \eqref{eq1.Alomari2017}, we get the desired result.
\end{proof}

An improvement of \eqref{eq6.Alomari2017} is given in the
following theorem.
\begin{corollary}
\label{cor1}Let $a,b\in\mathbb{R}$ with $b>a>0$. Let
$f:\left[a,b\right]\to\mathbb{R}$ be two absolutely continuous
functions such that  $f^{\prime},g^{\prime}\in L_2[a,b]$, then we
have
\begin{align}
\label{eq13.Alomari2017}
\left|{\widehat{\mathcal{P}}\left(f,g\right)}\right|
\le\frac{7}{{162 \pi ^2 }} \frac{1}{a^6 }\left( {b^9 - a^9 }
\right) \cdot \left\| {\ell f^{\prime} - f} \right\|_2\left\|
{\ell g^{\prime} - g} \right\|_2,
\end{align}
where $\ell\left(t\right)=t$, $t\in \left[a,b\right]$.
\end{corollary}

\begin{proof}
From \eqref{eq10.Alomari2017}, we have
\begin{align}
&\int_a^b {\int_a^b {\left( {tf\left( x \right) - xf\left( t
\right)} \right)^2  dx}dt}
\nonumber\\
&\le \frac{4}{{\pi ^2 }} \frac{1}{a^6 }  \left\| {\ell f^{\prime}
- f} \right\|_2^2\cdot\int_a^b {t^2\cdot \left[ {\frac{{b^3 - a^3
}}{6} + \frac{1}{3}\left| {t^3  - \frac{{a^3  + b^3 }}{2}}
\right|} \right]^2dt}
\nonumber\\
&\le \frac{4}{{\pi ^2 }} \frac{1}{a^6 }  \left\| {\ell f^{\prime}
- f} \right\|_2^2\cdot \left[\frac{7}{{324}}\left( {b^9  - a^9 }
\right) - \frac{7}{{108}}a^3 b^3 \left( {b^3  - a^3 }
\right)\right]
\nonumber\\
&\le \frac{4}{{\pi ^2 }} \frac{1}{a^6 }  \left\| {\ell f^{\prime}
- f} \right\|_2^2\cdot \frac{7}{{324}}\left( {b^9  - a^9 }
\right), \label{eq14.Alomari2017}
\end{align}
and similarly for $g$ we have
\begin{align}
\int_a^b{\int_a^b {\left( {tg\left( x \right) - xg\left( t
\right)} \right)^2  dx}dt}\le \frac{4}{{\pi ^2 }} \frac{1}{a^6 }
\left\| {\ell g^{\prime} - g} \right\|_2^2\cdot
\frac{7}{{324}}\left( {b^9 - a^9 }
\right).\label{eq15.Alomari2017}
\end{align}
Substituting \eqref{eq14.Alomari2017} and \eqref{eq15.Alomari2017}
in \eqref{eq7.Alomari2017} we get the desired result
\eqref{eq13.Alomari2017}.
\end{proof}

\begin{corollary}
 Let $a,b\in\mathbb{R}$ with $b>a>0$. Let
$a,b\in\mathbb{R}$ with $b>a>0$. Let
$f:\left[a,b\right]\to\mathbb{R}$ be an absolutely continuous
function such that $f^{\prime}\in L_2[a,b]$ and
$g:\left[a,b\right]\to\mathbb{R}$ is Lebesgue integrable on
$\left[a,b\right]$, then we have
\begin{align}
\left|{\widehat{\mathcal{P}}\left(f,g\right)}\right|\le
\frac{\sqrt{7}}{{9\sqrt{2} \pi   }} \frac{1}{a^3 }\left( {b^9 -
a^9 } \right)^{\frac{1}{2}}\left\| {\ell f^{\prime} - f}
\right\|_2\cdot
\left|{\widehat{\mathcal{P}}\left(g,g\right)}\right|^{\frac{1}{2}}.
\end{align}
\end{corollary}

\begin{corollary}
Let $a,b\in\mathbb{R}$ with $b>a>0$. Let
$f:\left[a,b\right]\to\mathbb{R}$ be an absolutely continuous
function such that  $f^{\prime}\in L^2[a,b]$ and
$g:\left[a,b\right]\to\mathbb{R}$ is measurable  on $[a,b]$ such
that $\gamma \le g\left(s\right)\le \Gamma$, for some real numbers
$\gamma,\Gamma$ and all $s\in \left[a,b\right]$, then the
inequality
\begin{multline}
\left|{\widehat{\mathcal{P}}\left(f,g\right)}\right|
\le\frac{\sqrt{7}}{{18\sqrt{6} \pi }} \frac{1}{a^3 }\left( {b^9 -
a^9 } \right)^{\frac{1}{2}} \cdot \left\| {\ell f^{\prime} - f}
\right\|_2
\\
\times \left[ {2\left( {b^3  - a^3 } \right)\left( {b - a}
\right)\left( {\Gamma ^2  + \gamma ^2 } \right) - 3\Gamma \gamma
\left( {b^2  - a^2 } \right)^2 } \right]^{\frac{1}{2}}
\end{multline}
holds.
\end{corollary}

\begin{proof}
Substituting \eqref{eq13.Alomari2017} in
\eqref{eq3.12.Alomari2017}, we get the desired result.
\end{proof}

%==========================================================================================================
%==========================================================================================================
\section{More inequalities}\label{sec3}
%==========================================================================================================
%==========================================================================================================
\begin{theorem}
\label{lem1.Alomari2017} Let $a,b\in\mathbb{R}$ with $b>a>0$. Let
$f:\left[a,b\right]\to\mathbb{R}$ be an absolutely continuous
functions  and $g:\left[a,b\right]\to\mathbb{R}$ is Lebesgue
integrable on $\left[a,b\right]$.
\begin{enumerate}
\item If $f^{\prime}\in L^\infty\left[a,b\right]$, then
\begin{align}
\label{eq4.1.Alomari2017}
\left|{\widehat{\mathcal{P}}\left(f,g\right)}\right|
\le\frac{1}{2} \left\| {f - \ell f'} \right\|_\infty \cdot
\int_a^b {\int_a^b {\left| {x - t} \right|\left| {tg\left( x
\right) - xg\left( t \right)} \right| dtdx} }.
\end{align}

\item If $f^{\prime}\in L^p\left[a,b\right]$, $p>1$  then
\begin{align}
\label{eq4.2.Alomari2017}
\left|{\widehat{\mathcal{P}}\left(f,g\right)}\right|
\le\frac{1}{2}\left( {\frac{1}{{2q - 1}}} \right)^{\frac{1}{q}}
\left\| {f - \ell f'} \right\|_p \cdot \int_a^b {\int_a^b {\left|
{\frac{{x^q }}{{t^{q - 1} }} - \frac{{t^q }}{{x^{q - 1} }}}
\right|^{\frac{1}{q}} \left| {tg\left( x \right) - xg\left( t
\right)} \right| dtdx} }.
\end{align}

\item If $f^{\prime}\in L^1\left[a,b\right]$, then

\begin{align}
\label{eq4.3.Alomari2017}
\left|{\widehat{\mathcal{P}}\left(f,g\right)}\right|
\le\frac{1}{2}\left\| {f - \ell f'} \right\|_1 \cdot \int_a^b
{\int_a^b {\frac{{\max \left\{ {x,t} \right\}}}{{\min \left\{
{x,t} \right\}}}  \left| {tg\left( x \right) - xg\left( t \right)}
\right| dtdx} }.
\end{align}
\end{enumerate}
$\ell\left(t\right)=t$, $t\in \left[a,b\right]$.
\end{theorem}

\begin{proof}
It is easy to observe  from Lemma \ref{Dragomir.lemma} that
\begin{enumerate}
\item If $f^{\prime}\in L^\infty\left[a,b\right]$, then
\begin{align*}
\left|{\widehat{\mathcal{P}}\left(f,g\right)}\right| &=\left|
{\int_a^b {\int_a^b {\left( {tf\left( x \right) - xf\left( t
\right)} \right)\left( {tg\left( x \right) - xg\left( t \right)}
\right)dtdx} } } \right|
\\
&\le  \frac{1}{2} \int_a^b {\int_a^b {\left| {tf\left( x \right) -
xf\left( t \right)} \right| \left| {tg\left( x \right) - xg\left(
t \right)} \right| dtdx} }
\\
&\le  \frac{1}{2} \left\| {f - \ell f'} \right\|_\infty \int_a^b
{\int_a^b {\left| {x - t} \right|\left| {tg\left( x \right) -
xg\left( t \right)} \right| dtdx} },
\end{align*}
which proves \eqref{eq4.1.Alomari2017}.\\

\item If $f^{\prime}\in L^p\left[a,b\right]$, $p>1$  then
\begin{align*}
\left|{\widehat{\mathcal{P}}\left(f,g\right)}\right| &=\left|
{\int_a^b {\int_a^b {\left( {tf\left( x \right) - xf\left( t
\right)} \right)\left( {tg\left( x \right) - xg\left( t \right)}
\right)dtdx} } } \right|
\\
&\le  \frac{1}{2} \int_a^b {\int_a^b {\left| {tf\left( x \right) -
xf\left( t \right)} \right| \left| {tg\left( x \right) - xg\left(
t \right)} \right| dtdx} }
\\
&\le \frac{1}{2} \left( {\frac{1}{{2q - 1}}} \right)^{\frac{1}{q}}
\left\| {f - \ell f'} \right\|_p \cdot \int_a^b {\int_a^b {\left|
{\frac{{x^q }}{{t^{q - 1} }} - \frac{{t^q }}{{x^{q - 1} }}}
\right|^{\frac{1}{q}} \left| {tg\left( x \right) - xg\left( t
\right)} \right| dtdx} },
\end{align*}
for $p,q>1$ with $\frac{1}{p}+\frac{1}{q}=1$, which proves \eqref{eq4.2.Alomari2017}.\\

\item If $f^{\prime}\in L^1\left[a,b\right]$, then
\begin{align*}
\left|{\widehat{\mathcal{P}}\left(f,g\right)}\right| &=\left|
{\int_a^b {\int_a^b {\left( {tf\left( x \right) - xf\left( t
\right)} \right)\left( {tg\left( x \right) - xg\left( t \right)}
\right)dtdx} } } \right|
\\
&\le  \frac{1}{2} \int_a^b {\int_a^b {\left| {tf\left( x \right) -
xf\left( t \right)} \right| \left| {tg\left( x \right) - xg\left(
t \right)} \right| dtdx} }
\\
&\le  \frac{1}{2}\left\| {f - \ell f'} \right\|_1 \cdot \int_a^b
{\int_a^b {\frac{{\max \left\{ {x,t} \right\}}}{{\min \left\{
{x,t} \right\}}}  \left| {tg\left( x \right) - xg\left( t \right)}
\right| dtdx} },
\end{align*}
which proves \eqref{eq4.3.Alomari2017},\\
\end{enumerate}

and this end the proof of the theorem.
\end{proof}

\begin{thm}
\label{thm8.Alomari2017} Let $a,b\in\mathbb{R}$ with $b>a>0$. Let
$f:\left[a,b\right]\to\mathbb{R}$ be an absolutely continuous
functions such that $f^{\prime}\in L_\infty\left[a,b\right]$ and
$g:\left[a,b\right]\to\mathbb{R}$ is measurable on
$\left[a,b\right]$  such that $\gamma \le g\left(s\right)\le
\Gamma$, for some  real numbers $\gamma,\Gamma $ and all $s\in
\left[a,b\right]$, then the inequality
\begin{align}
\label{eq4.4.Alomari2017}\left|{\widehat{\mathcal{P}}\left(f,g\right)}\right|
\le \frac{1}{12 }\left(b-a\right)^2 \left\|{f-\ell
f^{\prime}}\right\|_{\infty}\left[ {2\left( {b^3  - a^3 }
\right)\left( {b - a} \right)\left( {\Gamma ^2  + \gamma ^2 }
\right) - 3\Gamma \gamma \left( {b^2  - a^2 } \right)^2 } \right]
\end{align}
holds,  where $\ell\left(t\right)=t$, $t\in \left[a,b\right]$.
\end{thm}

\begin{proof}
From \eqref{eq1.Dragomir2015}, we have
\begin{align*}
 \left|{\widehat{\mathcal{P}}\left(f,f\right)}\right| \le
\frac{\left(b-a\right)^4}{12} \left\|{f-\ell
f^{\prime}}\right\|^2_{\infty}.
\end{align*}
Substituting in \eqref{eq3.12.Alomari2017}, we have
\begin{align*}
 \left|{\widehat{\mathcal{P}}\left(f,g\right)}\right|&\le
\frac{1}{2\sqrt{3}}\left|{\widehat{\mathcal{P}}\left(f,f\right)}\right|^{\frac{1}{2}}
\left[ {2\left( {b^3  - a^3 } \right)\left( {b - a} \right)\left(
{\Gamma ^2  + \gamma ^2 } \right) - 3\Gamma \gamma \left( {b^2  -
a^2 } \right)^2 } \right]^{\frac{1}{2}}
\\
&\le \frac{1}{12}\left(b-a\right)^2 \left\|{f-\ell
f^{\prime}}\right\|_{\infty}\left[ {2\left( {b^3  - a^3 }
\right)\left( {b - a} \right)\left( {\Gamma ^2  + \gamma ^2 }
\right) - 3\Gamma \gamma \left( {b^2  - a^2 } \right)^2 }
\right]^{\frac{1}{2}},
\end{align*}
and this proves the required inequality.
\end{proof}

\begin{thm}
\label{thm10.Alomari2017} Let $a,b\in\mathbb{R}$ with $b>a>0$. Let
$f,g:\left[a,b\right]\to\mathbb{R}$ be two absolutely continuous
functions such that $f^{\prime}\in L_\infty\left[a,b\right]$ and
$g^{\prime}\in L_1\left[a,b\right]$, then the inequality then we
have
\begin{align}
\label{eq4.6.Alomari2017}\left|{\widehat{\mathcal{P}}\left(f,g\right)}\right|
\le\frac{1}{2}\left(b^2 a - \frac{{a^3  + 8b^3 }}{9} +
\frac{2}{3}b^3 \ln \left( {\frac{b}{a}} \right) \right)\left\| {f
- \ell f^{\prime}} \right\|_\infty \left\| {g - \ell g^{\prime}}
\right\|_1
\end{align}
where $\ell\left(t\right)=t$, $t\in \left[a,b\right]$.

\end{thm}

\begin{proof}
Using \eqref{eq.Dragomir.lemma}, we have
\begin{align}
 \left|{\widehat{\mathcal{P}}\left(f,g\right)}\right|
&=\left| {\int_a^b {\int_a^b {\left( {tf\left( x \right) -
xf\left( t \right)} \right)\left( {tg\left( x \right) - xg\left( t
\right)} \right)dx } dt} } \right|
\nonumber \\
&\le  \frac{1}{2} \int_a^b {\int_a^b {\left| {tf\left( x \right) -
xf\left( t \right)} \right| \left| {tg\left( x \right) - xg\left(
t \right)} \right| dx}dt }
\nonumber \\
&\le  \frac{1}{2} \left\| {f - \ell f'} \right\|_\infty \left\| {g
- \ell g'} \right\|_1\int_a^b {\int_a^b {\left| {x - t}
\right|\cdot \frac{{\max \left\{ {x,t} \right\}}}{{\min \left\{
{x,t} \right\}}}   dx} dt}.  \label{eq4.7.Alomari2017}
\end{align}
Now,
\begin{align*}
 &\int_a^b {\int_a^b {\left| {x - t} \right| \cdot \frac{{\max \left\{ {x,t} \right\}}}{{\min \left\{
{x,t} \right\}}} dxdt} }  \\
  &= \int_a^b {\left[ {\int_a^t {\left| {x - t} \right| \cdot \frac{{\max \left\{ {x,t} \right\}}}{{\min \left\{
{x,t} \right\}}}dx}  + \int_t^b {\left| {x - t} \right| \cdot
\frac{{\max \left\{ {x,t} \right\}}}{{\min \left\{
{x,t} \right\}}}dx} } \right]dt}  \\
  &= \int_a^b {\left[ {\int_a^t {\left( {t - x} \right) \cdot \frac{t}{x}dx}  + \int_t^b {\left( {x - t} \right) \cdot \frac{x}{t}dx} } \right]dt}  \\
  &= \int_a^b {\left[ {
t^2 \left( {\ln \left( t \right) - \ln \left( a \right)} \right) -
t\left( {t - a} \right) + \frac{{b^3 }}{{3t}} - \frac{{t^2 }}{3} -
\frac{{b^2 }}{2} + \frac{{t^2 }}{2}} \right]dt}  \\
  &=
b^2 a - \frac{{a^3  + 8b^3 }}{9} + \frac{2}{3}b^3 \ln \left(
{\frac{b}{a}} \right),
 \end{align*}
substituting in \eqref{eq4.7.Alomari2017}, we get the desired
result.
\end{proof}

\begin{thm}
\label{thm12.Alomari2017} Let $a,b\in\mathbb{R}$ with $b>a>0$. Let
$f:\left[a,b\right]\to\mathbb{R}$ be an absolutely continuous
functions such that $f^{\prime}\in L_1\left[a,b\right]$ and
$g:\left[a,b\right]\to\mathbb{R}$ is measurable  on
$\left[a,b\right]$ such that $\gamma \le g\left(s\right)\le
\Gamma$ for some real numbers $\gamma,\Gamma $, $s\in
\left[a,b\right]$, then the inequality
\begin{multline*}
\left|{\widehat{\mathcal{P}}\left(f,g\right)}\right|
\\
\le \frac{1}{6\sqrt{2}}\left(\frac{{2b^3  + a^3  - 3ab^2 }}{{
a}}\right)^{\frac{1}{2}} \left\|{f-\ell f^{\prime}}\right\|_{1}
\left[ {2\left( {b^3  - a^3 } \right)\left( {b - a} \right)\left(
{\Gamma ^2  + \gamma ^2 } \right) - 3\Gamma \gamma \left( {b^2  -
a^2 } \right)^2 } \right]^{\frac{1}{2}}
\end{multline*}
holds,  where $\ell\left(t\right)=t$, $t\in \left[a,b\right]$.
\end{thm}

\begin{proof}
From \eqref{eq4.Dragomir2015}, we have
\begin{align*}
\left|{\widehat{\mathcal{P}}\left(f,f\right)}\right| \le
\frac{{2b^3  + a^3  - 3ab^2 }}{{6a}}\left\| {f - \ell f'}
\right\|_1^2.
\end{align*}
Substituting in \eqref{eq3.12.Alomari2017}, we have
\begin{align*}
 \left|{\widehat{\mathcal{P}}\left(f,g\right)}\right|&\le
\frac{1}{2\sqrt{3}}\left|{\widehat{\mathcal{P}}\left(f,f\right)}\right|^{\frac{1}{2}}
\left[ {2\left( {b^3  - a^3 } \right)\left( {b - a} \right)\left(
{\Gamma ^2  + \gamma ^2 } \right) - 3\Gamma \gamma \left( {b^2  -
a^2 } \right)^2 } \right]^{\frac{1}{2}}
\\
&\le \frac{1}{6\sqrt{2}}\left(\frac{{2b^3  + a^3  - 3ab^2 }}{{
a}}\right)^{\frac{1}{2}} \left\|{f-\ell f^{\prime}}\right\|_{1}
\\
&\qquad\times\left[ {2\left( {b^3  - a^3 } \right)\left( {b - a}
\right)\left( {\Gamma ^2  + \gamma ^2 } \right) - 3\Gamma \gamma
\left( {b^2  - a^2 } \right)^2 } \right]^{\frac{1}{2}},
\end{align*}
and this proves the theorem.
\end{proof}
%==========================================================================================================
%==========================================================================================================
\section{Generalizations, Remarks and Conclusion}\label{sec4}
%==========================================================================================================
%==========================================================================================================
\subsection{Boggio MVT}The MVT of  Pompeiu was generalized by Boggio
\cite{B} in 1947 (see also, p.,92; \cite{Sahoo}), where he proved
the following generalization of PMVT:
\begin{theorem}
\label{thm.Boggio}For every real valued functions $f$ and $h$
differentiable on an interval $[a,b]$ not containing $0$ and for
all pairs $x_1\ne x_2$ in $[a,b]$ there exists a point $\xi\in
(x_1,x_2)$ such that
\begin{align}
\label{eq.Boggio}\frac{{h\left( {x_1 } \right)f\left( {x_2 }
\right) -h\left( {x_2 } \right)f\left( {x_1 } \right)}}{{h\left(
{x_1 } \right) - h\left( {x_2 } \right)}} = f\left( \xi  \right) -
\frac{{h\left( \xi  \right)}}{{h^{\prime}\left( \xi
\right)}}f^{\prime}\left( \xi \right)\text{''.}
\end{align}
\end{theorem}

 As we notice, Boggio generalization of PMVT    deals with real functions
defined on   real intervals  not containing $`0$'. The natural
question is:   Is it necessarily  to exclude $`0$' from
$\left[a,b\right]$ in the Boggio generalization?.

The answer is not necessary to exclude $`0$' from
$\left[a,b\right]$. For example, let $h\left(x\right)= x^2-4x+4$,
$x\in \left[-1,1\right]$. Clearly, $h$ is differentiable and
$h\left(x\right)\ne0\ne h^{\prime}\left(x\right)$ for all
$x\in\left[-1,1\right]$ and $0\in \left[-1,1\right]$. The function
$h$ satisfies the assumptions of Theorem \ref{thm.Boggio}, for
instance, choose $f\left(x\right)=x$, $x\in \left[-1,1\right]$.
Applying \eqref{eq.Boggio}, for $x_1=-1$ and $x_2=1$, we get
\begin{align*}
\frac{{9\left( 1 \right) - 1\left( { - 1} \right)}}{{9 - 1}} =
\frac{5}{4} = \xi  - \frac{{\xi ^2  - 4\xi  + 4}}{{2\xi  - 4}}.
\end{align*}
Solving for $\xi$, we get $\xi = \frac{1}{2} \in
\left(-1,1\right)$. But $0\in \left[-1,1\right]$\,\,(?)!. Which
means that it's not necessary   to exclude $`0$' from
$\left[a,b\right]$.

To deal with more large class of functions and intervals, we need
to revise Theorem \ref{thm.Boggio}. For more details about several
and various type of MVTs
and their generalization(s) we refer the reader to \cite{Sahoo}.\\

Let $a,b\in \mathbb{R}$, and $I$ be a real interval such that
$a,b$ belong to $I^{\circ}$; the interior of $I$ with $a<b$. Let
$\mathcal{A}$ be the set of all real intervals $I$ for which
neither $h\left(x\right)$ nor $h^{\prime}\left(x\right)$ is ever
zero on $\mathcal{A}$, where $h:I\to\mathbb{R}$ is a real valued
differentiable function. In symbols, we may write
\begin{align*}
\mathcal{A} :=\{{[a,b]: a,b \in I^{\circ}, h\left(x\right)\ne 0
\ne h^{\prime}\left(x\right), \forall x\in [a,b]}\}.
\end{align*}
In what follows, we revise Theorem \ref{thm.Boggio} and present a
new independent proof.
\begin{theorem}
Let $[a,b] \in \mathcal{A}$. For every real valued differentiable
functions $f$ and $h$ defined  on  $[a,b]$ and  all distinct pairs
$x_1, x_2 \in [a,b]$ with $h\left( {x_1 } \right) \ne h\left( {x_2
} \right)$, there exists a point $\xi\in (x_1,x_2)$ such that
\eqref{eq.Boggio} holds.
\end{theorem}
\begin{proof}
Let $x_1,x_2$ be any two distinct points in $[a,b]$ with $x_1<x_2$
and $h\left( {x_1 } \right) \ne h\left( {x_2 } \right)$. Define
the function $F:[x_1,x_2]\to \mathbb{R}$, given by
\begin{align*}
F\left( x \right) = \left[ {h\left( {x_1 } \right) - h\left( {x_2
} \right)} \right]\frac{{f\left( x \right)}}{{h\left( x \right)}}
- \frac{1}{{h\left( x \right)}}\left[ {h\left( {x_1 }
\right)f\left( {x_2 } \right) - h\left( {x_2 } \right)f\left( {x_1
} \right)} \right].
\end{align*}
Clearly, $F$ is continuous on $[x_1,x_2]$, differentiable on
$(x_1,x_2)$ and
\begin{align*}
F\left( {x_1 } \right) &= \left[ {h\left( {x_1 } \right) - h\left(
{x_2 } \right)} \right]\frac{{f\left( {x_1 } \right)}}{{h\left(
{x_1 } \right)}} - \frac{1}{{h\left( {x_1 } \right)}}\left[
{h\left( {x_1 } \right)f\left( {x_2 } \right) - h\left( {x_2 }
\right)f\left( {x_1 } \right)} \right]
\\
&= f\left( {x_1 } \right) - f\left( {x_2 } \right)
\\
&=  F\left( {x_2 } \right).
\end{align*}
Applying Rolle's theorem, there is an $\xi \in (x_1,x_2)$ such
that  $F^{\prime}\left( \xi \right) = 0$, so that
\begin{align*}
F^{\prime}\left( \xi  \right) &= \left[ {h\left( {x_1 } \right)
-h\left( {x_2 } \right)} \right]\frac{{h\left( \xi
\right)f^{\prime}\left( \xi \right) - f\left( \xi
\right)h^{\prime}\left( \xi \right)}}{{h^2 \left( \xi  \right)}}
\\
&\qquad- \frac{{h^{\prime}\left( \xi  \right)}}{{h^2 \left( \xi
\right)}}\left[ {h\left( {x_1 } \right)f\left( {x_2 } \right) -
h\left( {x_2 } \right)f\left( {x_1 } \right)} \right]
\\
&=0,
\end{align*}
equivalently we write
\begin{align*}
\left[ {h\left( {x_1 } \right) - h\left( {x_2 } \right)}
\right]h\left( \xi  \right)f^{\prime}\left( \xi  \right) &=
h^{\prime}\left( \xi \right)\left[ {h\left( {x_1 } \right)f\left(
{x_2 } \right) - h\left( {x_2 } \right)f\left( {x_1 } \right)}
\right]
\\
&\qquad+ \left[ {h\left( {x_1 } \right) -h\left( {x_2 } \right)}
\right]f\left( \xi  \right)h^{\prime}\left( \xi  \right),
\end{align*}
which means that
\begin{align*}
\frac{{h\left( {x_1 } \right)f\left( {x_2 } \right) - h\left( {x_2
} \right)f\left( {x_1 } \right)}}{{h\left( {x_1 } \right) -
h\left( {x_2 } \right)}}=f\left( \xi  \right)- h\left( \xi
\right)\frac{{f^{\prime}\left( \xi \right)}}{{h^{\prime}\left( \xi
\right)}},
\end{align*}
and the proof is established.
\end{proof}

\begin{remark}
Let $a,b\in\mathbb{R}$ with $b>a>0$. By setting
$h\left(x\right)=x$, we refer to the {\rm{PMVT}}. More generally,
for $h\left(x\right)=x^r$, $r\in \mathbb{R}-\{0\}$, then
\eqref{eq.Boggio} becomes
\begin{align*}
\frac{{x^r_1f\left( {x_2 } \right) - x^r_2 f\left( {x_1 }
\right)}}{{x^r_1  - x^r_2}}=f\left( \xi \right)- \frac{\xi}{r
}f^{\prime}\left( \xi \right).
\end{align*}
For all distinct pairs $x_1, x_2 \in [a,b]$.

This type of MVT was applied to obtain Ostrowski's type
inequalities in \cite{Acu1}, \cite{Acu2}, \cite{Dragomir2} and
\cite{Popa}. For comprehensive list of results regarding
Ostrowski's inequality see the recent survey \cite{Dragomir}.
\end{remark}

\subsection{Pompeiu--Chebyshev functional}
Let $a,b\in \mathbb{R}$, $a<b$. Let $f,g,h:[a,b]\to \mathbb{R}$ be
three integrable functions, then the Pompeiu--Chebyshev functional
can be introduced such as:
\begin{align}
\label{Pompeiu.Chebyshev.Identity}\widehat{\mathcal{P}}_h\left(f,g\right)
&=\int_a^b {h^2 \left( x \right)dx} \int_a^b {f\left( t
\right)g\left( t \right)dt}  - \int_a^b {f\left( t \right)h\left(
t \right)dt} \int_a^b {h\left( x \right)g\left( x \right)dx}
\nonumber\\
&= \frac{1}{2}\int_a^b {\int_a^b {\left( {h\left( x \right)f\left(
t \right) - h\left( t \right)f\left( x \right)} \right)\left(
{h\left( x \right)g\left( t \right) - h\left( t \right)g\left( x
\right)} \right)dt} dx}.
\end{align}
If we consider $h\left( x \right)=1$, then
$\widehat{\mathcal{P}}_1\left(f,g\right) = \left( b-a
\right)\mathcal{T}\left(f,g\right)$,  which is the Chebyshev
functional \eqref{eq1.1.5}. Also, if $h\left( x \right)=x$, $x\in
[a,b]$, $b>a>0$, then $\widehat{\mathcal{P}}_x\left(f,g\right)
=\widehat{\mathcal{P}}\left(f,g\right) $, which is studied in the Sections \ref{sec2} and \ref{sec3}.\\

After we proposed $\widehat{\mathcal{P}}_h\left(f,g\right)$
independently, we noticed that
$\widehat{\mathcal{P}}_h\left(f,g\right)$ could be deduced from
more general identity of Andrei\'{e}f's (see \cite{MPF}, p.243),
which reads: For two continuous functions $f$ and $g$ defined on
$[a,b]$, we have the representation:
\begin{multline}
\label{eq1.2.2}\int_a^b {F_1 \left( x \right)F_2 \left( x
\right)dx} \int_a^b {G_1 \left( x \right)G_2 \left( x \right)dx} -
\int_a^b {F_1 \left( x \right)G_2 \left( x \right)dx} \int_a^b
{F_2 \left( x \right)G_1 \left( x \right)dx}
\\
= \frac{1}{2}\int_a^b {\int_a^b {\left( {F_1 \left( x \right)G_1
\left( y \right) - F_1 \left( y \right)G_1 \left( x \right)}
\right)\left( {F_2 \left( x \right)G_2 \left( y \right) - F_2
\left( y \right)G_2 \left( x \right)} \right)dx} dy}.
\end{multline}
Simply, substituting  $F_1(x)=F_2(x)=h(x)$,  $G_1(x)=f(x)$, and
$G_2(x)=g(x)$ in  \eqref{eq1.2.2}, then we obtain
\eqref{Pompeiu.Chebyshev.Identity}.

\begin{lemma}
\label{lemma2} Let $f,h : [a, b]\to \mathbb{R}$ be two absolutely
continuous function on the interval $[a, b]$. Then for any $t, x
\in [a, b]$, we have
\begin{align}
\left| {h\left( t \right)f\left( x \right) - h\left( x
\right)f\left( t \right)} \right|
  \le \left| {h\left( x \right)} \right|\left| {h\left( t \right)}
\right|\left\{ \begin{array}{l}
 \left| {hf' - fh'} \right|_{\infty ,\left[ {t,x} \right]} \int_t^x {\frac{{ds}}{{\left| {h\left( s \right)} \right|^2 }}}  \\
 \\
 \left| {hf' - fh'} \right|_{p,\left[ {t,x} \right]} \left( {\int_t^x {\frac{{ds}}{{\left| h\left( s \right) \right|^{2q}}}} } \right)^{\frac{1}{q}}  \\
 \\
 \mathop {\sup }\limits_{s \in \left[ {t,x} \right]} \left\{ {\frac{1}{{\left| {h\left( s \right)} \right|^2 }}} \right\}\left| {hf' - fh'} \right|_{1,\left[ {t,x} \right]}  \\
 \end{array} \right.
 \end{align}
where $p,q>1$ with $\frac{1}{p}+\frac{1}{q}=1$. Provided that
$h(s)\ne 0$, for $s\in [a,b]$.
\end{lemma}

\begin{proof}
Since $f$ and $h$ are absolutely continuous functions on $[a,b]$,
then $\frac{f\left( s \right)}{h\left( s \right)}$ is absolutely
continuous on $[a,b]$ and so that
\begin{align*}
\int_t^x {\left( {\frac{{f\left( s \right)}}{{h\left( s \right)}}}
\right)^\prime  ds}  = \frac{{f\left( x \right)}}{{h\left( x
\right)}} - \frac{{f\left( t \right)}}{{h\left( t \right)}},
\end{align*}
for any $t, x \in [a,b]$ with $x \ne t$.

Since
\begin{align*}
\int_t^x {\left( {\frac{{f\left( s \right)}}{{h\left( s \right)}}}
\right)^\prime  ds}  = \frac{{h\left( s \right)f'\left( s \right)
- f\left( s \right)h'\left( s \right)}}{{h^2 \left( s \right)}},
\end{align*}
we get the identity
\begin{align*}
h\left( t \right)f\left( x \right) - h\left( x \right)f\left( t
\right) = h\left( x \right)h\left( t \right)\int_t^x
{\frac{{h\left( s \right)f'\left( s \right) - f\left( s
\right)h'\left( s \right)}}{{h^2 \left( s \right)}}ds}.
\end{align*}
Taking the modulus, we have
\begin{align*}
&\left| {h\left( t \right)f\left( x \right) - h\left( x
\right)f\left( t \right)} \right|
\\
&= \left| {h\left( x \right)} \right|\left| {h\left( t \right)}
\right|\left| {\int_t^x {\frac{{h\left( s \right)f'\left( s
\right) - f\left( s \right)h'\left( s \right)}}{{h^2 \left( s
\right)}}ds} } \right|
\\
&\le \left| {h\left( x \right)} \right|\left| {h\left( t \right)}
\right|\int_t^x {\left| {\frac{{h\left( s \right)f'\left( s
\right) - f\left( s \right)h'\left( s \right)}}{{h^2 \left( s
\right)}}} \right|ds}
\\
&\le \left| {h\left( x \right)} \right|\left| {h\left( t \right)}
\right|\left\{ \begin{array}{l}
 \mathop {\sup }\limits_{s \in \left[ {t,x} \right]} \left| {h\left( s \right)f'\left( s \right) - f\left( s \right)h'\left( s \right)} \right|\int_t^x {\frac{{ds}}{{\left| {h\left( s \right)} \right|^2 }}}  \\
 \\
 \left( {\int_x^t {\left| {h\left( s \right)f'\left( s \right) - f\left( s \right)h'\left( s \right)} \right|^p ds} } \right)^{\frac{1}{p}} \left( {\int_t^x {\frac{{ds}}{{\left| {h\left( s \right)} \right|^{2q}}}} } \right)^{\frac{1}{q}}  \\
 \\
 \mathop {\sup }\limits_{s \in \left[ {t,x} \right]} \left\{ {\frac{1}{{\left| {h\left( s \right)} \right|^2 }}} \right\}\int_t^x {\left| {h\left( s \right)f'\left( s \right) - f\left( s \right)h'\left( s \right)} \right|ds}  \\
 \end{array} \right.
\\
&= \left| {h\left( x \right)} \right|\left| {h\left( t \right)}
\right|\left\{ \begin{array}{l}
 \left| {hf' - fh'} \right|_{\infty ,\left[ {t,x} \right]} \int_t^x {\frac{{ds}}{{\left| {h\left( s \right)} \right|^2 }}}  \\
 \\
 \left| {hf' - fh'} \right|_{p,\left[ {t,x} \right]} \left( {\int_t^x {\frac{{ds}}{{\left| {h\left( s \right)} \right|^{2q} }}} } \right)^{\frac{1}{q}}  \\
 \\
 \mathop {\sup }\limits_{s \in \left[ {t,x} \right]} \left\{ {\frac{1}{{\left| {h\left( s \right)} \right|^2 }}} \right\}\left| {hf' - fh'} \right|_{1,\left[ {t,x} \right]}  \\
 \end{array} \right.,
\end{align*}
which completes the proof.
\end{proof}

\begin{remark}
In the previous lemma, if we choose $h(s)=s^r$, $r \ge 1$,
$s\in[a,b]$, $b>a>0$. we have
\begin{align*}
\left| {t^r f\left( x \right) - x^r f\left( t \right)} \right| =
\left\{ \begin{array}{l}
 \frac{1}{{2r - 1}} \cdot \left| {\frac{{t^r }}{{x^{r - 1} }} - \frac{{x^r }}{{t^{r - 1} }}} \right|\left| {\ell _r f' - r\ell _{r - 1} f} \right|_{\infty ,\left[ {t,x} \right]}  \\
 \\
 \left( {\frac{1}{{2rq - 1}}} \right)^{\frac{1}{q}} \left| {\frac{{t^{rq} }}{{x^{rq - 1} }} - \frac{{x^{rq} }}{{t^{rq - 1} }}} \right|^{\frac{1}{q}} \left| {\ell _r f' - r\ell _{r - 1} f} \right|_{p,\left[ {t,x} \right]}  \\
 \\
 \left( {\frac{{\max \left\{ {t,x} \right\}}}{{\min \left\{ {t,x} \right\}}}} \right)^r \left| {\ell _r f' - r\ell _{r - 1} f} \right|_{1,\left[ {t,x} \right]}  \\
 \end{array} \right.,
\end{align*}
for $p,q>1$ with $\frac{1}{p}+\frac{1}{q}=1$, where $\ell
_r(s)=s^r$, $s\in[a,b]$ and  $r \ge 1$.
\end{remark}

\begin{remark}
By following the same approaches considered in \cite{Dragomir1}
and  in the Sections \ref{sec2} and \ref{sec3} of this work one
can state more general results concerning
$\widehat{\mathcal{P}}_h\left(f,g\right)$.  We left this part to
the interested reader and focused researchers.
\end{remark}

\begin{definition}
A real valued function $f$ defined on $\left[a,b\right]$ is called
increasing (decreasing) with respect to a non-negative function
$h:[a,b]\to \mathbb{R}_+$ or simply $h$-increasing
($h$-decreasing) if and only if
\begin{align*}
h\left( x \right)f\left( t \right) - h\left( t \right)f\left( x
\right) \ge (\le)\,\, 0
\end{align*}
whenever $t \ge  x$, for every $x,t \in [a,b]$. In special case if
$h(x)=1$ we refer to the original monotonicity. Also, if  $h(x)=x$
we have
\begin{align*}
xf\left( t \right) - t f\left( x
\right) \ge (\le)\,\, 0,
\end{align*}
which used in Theorem
\ref{thm1.1.1}.
\end{definition}

Next result generalize the inequality \eqref{eq3.1.Alomari2017}
and the Chebyshev first inequality (see
\eqref{Chebyshev.first.inequality}):
\begin{thm}
\label{thmxx.Alomari2017}Let $a,b\in\mathbb{R}$ with $a<b$. Let
$f,g: \left[ {a,b} \right]\subseteq \mathbb{R}\to \mathbb{R}$ be
three integrable functions on $\left[ {a,b} \right]$. If
$h:[a,b]\to \mathbb{R}_+$ is integrable on $[a,b]$ and  $f$ and
$g$ are both $h$-increasing or $h$-decreasing on $[a,b]$, then
\begin{align}
\label{xx.Alomari2017} \int_a^b {h^2 \left( x \right)dx} \int_a^b
{f\left( x \right)g\left(x \right)dx}  \ge\int_a^b {f\left( x
\right)h\left( x \right)dx} \int_a^b {h\left( x \right)g\left( x
\right)dx}.
\end{align}
\end{thm}

\begin{proof}
Assume If $f$ and $g$ are both $h$-increasing on $[a,b]$, then we
have
\begin{align*}
\left[h\left( x \right)f\left( y \right) - h\left(y \right)f\left(
x \right) \right] \left[h\left( x \right)g\left( y \right) -
h\left( y\right)g\left( x \right) \right]\ge  0,
\end{align*}
for all $x, y\in [a,b]$, therefore we have
\begin{align*}
\int_a^b {\int_a^b {\left[h\left( x \right)f\left( y \right) -
h\left(y \right)f\left( x \right) \right] \left[h\left( x
\right)g\left( y \right) - h\left( y\right)g\left( x \right)
\right]dx}dy }\ge 0.
\end{align*}
But also we have
\begin{align*}
&\frac{1}{2}\int_a^b {\int_a^b {\left[h\left( x \right)f\left( y
\right) - h\left(y \right)f\left( x \right) \right] \left[h\left(
x \right)g\left( y \right) - h\left( y\right)g\left( x \right)
\right]dx}dy }
\\
&=  \int_a^b {h^2 \left( x \right)dx} \int_a^b {f\left( x
\right)g\left( x \right)dx}- \int_a^b {f\left( x \right)h\left( x
\right)dx} \int_a^b {h\left( x \right)g\left( x \right)dx}
\\
&\ge 0,
\end{align*}
which proves the inequality \eqref{xx.Alomari2017}.
\end{proof}

\begin{remark}
In \eqref{xx.Alomari2017}, if $h\left(t\right)=1$,  $t\in [a,b]$,
we recapture the first Chebyshev inequality, which reads:
\begin{align}
\label{Chebyshev.first.inequality}\left( b-a \right) \int_a^b
{f\left( x \right)g\left(x \right)dx} \ge\int_a^b {f\left( x
\right)dx} \int_a^b {g\left( x \right)dx}.
\end{align}
If $h\left(t\right)=t$,  $t\in [a,b]$, then  we recapture
\eqref{eq3.1.Alomari2017}.
\end{remark}

\begin{corollary}
\label{thm.cor.important}Let $a,b\in\mathbb{R}$ with $a<b$. Let
$f,g,h: \left[ {a,b} \right]\subseteq \mathbb{R}\to \mathbb{R}$ be
two integrable functions on $\left[ {a,b} \right]$. If $h:[a,b]\to
\mathbb{R}_+$ is integrable on $[a,b]$, $f$ is $h$-increasing and
$g$ is $h$-decreasing on $[a,b]$, then
\begin{align}
\int_a^b {h^2 \left( x \right)dx} \int_a^b {f\left( x
\right)g\left(x \right)dx}  \le\int_a^b {f\left( x \right)h\left(
x \right)dx} \int_a^b {h\left( x \right)g\left( x \right)dx}.
\end{align}
\end{corollary}
\begin{proof}
The proof goes likewise the proof of Theorem
\ref{thmxx.Alomari2017}.
\end{proof}

A generalization of Theorem  \ref{thm1.Alomari2017}, the
celebrated pre-Gr\"{u}ss inequality is incorporated in the
following theorem.
\begin{corollary}
Let $a,b\in\mathbb{R}$ with $a<b$. Let $f,g: \left[ {a,b}
\right]\subseteq \mathbb{R}\to \mathbb{R}$ be two integrable
functions on $\left[ {a,b} \right]$. If $h:[a,b]\to \mathbb{R}_+$
is integrable on $[a,b]$, then
\begin{align}
\label{gene.pre.gruss2017}\left|{\widehat{\mathcal{P}}_h\left(f,g\right)}\right|
\le
\left|{\widehat{\mathcal{P}}_h\left(f,f\right)}\right|^{\frac{1}{2}}
\left|{\widehat{\mathcal{P}}_h\left(g,g\right)}\right|^{\frac{1}{2}},
\end{align}
or equivalently we can write
\begin{align*}
\left|{\widehat{\mathcal{P}}_h\left(f,g\right)}\right| \le
\left|{\widehat{\mathcal{P}}_f\left(h,h\right)}\right|^{\frac{1}{2}}
\left|{\widehat{\mathcal{P}}_g\left(h,h\right)}\right|^{\frac{1}{2}}.
\end{align*}
Both inequalities are sharp.
\end{corollary}

\begin{proof}
Since
\begin{align*}
\left|{\widehat{\mathcal{P}}_h\left(f,g\right)}\right| =
\left|{\frac{1}{2}\int_a^b {\int_a^b {\left[h\left( x
\right)f\left( y \right) - h\left(y \right)f\left( x \right)
\right] \left[h\left( x \right)g\left( y \right) - h\left(
y\right)g\left( x \right) \right]dx}dy }}\right|
\end{align*}
Utilizing the triangle inequality and then the
Cauchy-Bunyakovsky-Schwarz (CBS) inequality, we get
\begin{align}
\left|{\widehat{\mathcal{P}}_h\left(f,g\right)}\right| &\le
\frac{1}{2}\int_a^b {\int_a^b {\left|h\left( x \right)f\left( y
\right) - h\left(y \right)f\left( x \right) \right| \left|h\left(
x \right)g\left( y \right) - h\left( y\right)g\left( x \right)
\right|dx}dy }
\nonumber\\
&\le   \left( {\frac{1}{2}\int_a^b {\int_a^b {\left( {h\left( x
\right)f\left( y \right) - h\left( y \right)f\left( x \right)}
\right)^2 dy} dx} } \right)^{\frac{1}{2}} \label{Gene.Cheb.POMP}
\\
&\qquad\times  \left( {\frac{1}{2}\int_a^b {\int_a^b {\left(
{h\left( x \right)g\left( y \right) - h\left( y \right)g\left( x
\right)} \right)^2 dy} dx} } \right)^{\frac{1}{2}}.\nonumber
\end{align}
On the other hand, we have
\begin{align*}
  \frac{1}{2}\int_a^b {\int_a^b {\left( {h\left( x \right)f\left( y \right) - h\left( y \right)f\left( x \right)} \right)^2 dy} dx}
     &=
\int_a^b {h^2 \left( x \right)dx} \int_a^b {f^2 \left( y
\right)dy}  - \left( {\int_a^b {h\left( x \right)f\left( x
\right)dx} } \right)^2
\\
&=\widehat{\mathcal{P}}_h\left(f,f\right)\equiv
\widehat{\mathcal{P}}_f\left(h,h\right),
\end{align*}
and similarly,
\begin{align*}
  \frac{1}{2}\int_a^b {\int_a^b {\left( {h\left( x \right)g\left( y \right) - h\left( y \right)g\left( x \right)} \right)^2 dy}
 dx}
  &=
\int_a^b {h^2 \left( x \right)dx} \int_a^b {g^2 \left( y
\right)dy}  - \left( {\int_a^b {h\left( x \right)g\left( x
\right)dx} } \right)^2
\\
&=\widehat{\mathcal{P}}_h\left(g,g\right)\equiv
\widehat{\mathcal{P}}_g\left(h,h\right).
\end{align*}
Substituting in \eqref{Gene.Cheb.POMP}, we get  the required
result.
\end{proof}
\begin{remark}
In \eqref{gene.pre.gruss2017}, if $h\left(t\right)=1$,  $t\in
[a,b]$, we recapture the classical version of pre-Gr\"{u}ss
inequality, which reads:
\begin{align*}
\left|{\mathcal{T}\left(f,g\right)}\right| \le
\left|{\mathcal{T}\left(f,f\right)}\right|^{\frac{1}{2}}
\left|{\mathcal{T}\left(g,g\right)}\right|^{\frac{1}{2}}.
\end{align*}
If $h\left(t\right)=t$,  $t\in [a,b]$, then  we recapture
\eqref{eq1.Alomari2017}.
\end{remark}

\begin{remark}
The weighted version of $\widehat{\mathcal{P}}_h\left(f,g\right)$
can be presented using Andrei\'{e}f's weighted version of
\eqref{eq1.2.2}  \emph{(see \cite{MPF}, p.244)}, which reads: For
arbitrary continuous functions $F_i,G_i:[a,b]\to \mathbb{R}$
$(i=1,2)$ defined on $[a,b]$ and a positive continuous function
$p:[a,b]\to \mathbb{R}_+$, we have
\begin{align}
\label{eq1.2.4}&\left| {\begin{array}{*{20}c}
   {\int_a^b {p\left( x \right)F_1 \left( x \right)G_1 \left( x \right)dx} } & {\int_a^b {p\left( x \right)F_1 \left( x \right)G_2 \left( x \right)dx} }  \\
   {\int_a^b {p\left( x \right)F_2 \left( x \right)G_1 \left( x \right)dx} } & {\int_a^b {p\left( x \right)F_2 \left( x \right)G_2 \left( x \right)dx} }  \\
\end{array}} \right|
\\
& = \frac{1}{2}\int_a^b {\int_a^b {\left| {\begin{array}{*{20}c}
   {F_1 \left( {x_1 } \right)} & {F_1 \left( {x_2 } \right)}  \\
   {F_2 \left( {x_1 } \right)} & {F_2 \left( {x_2 } \right)}  \\
\end{array}} \right|\left| {\begin{array}{*{20}c}
   {G_1 \left( {x_1 } \right)} & {G_1 \left( {x_2 } \right)}  \\
   {G_2 \left( {x_1 } \right)} & {G_2 \left( {x_2 } \right)}  \\
\end{array}} \right|p\left( {x_1 } \right)p\left( {x_2 } \right)dx_1 dx_2 }
}.\nonumber
\end{align}
For $p(x)=1$, we get the original version of Andrei\'{e}f's
identity \eqref{eq1.2.2}. Moreover, by substituting $F_1\left( {t}
\right)=G_1\left( {t} \right)=h\left( {t} \right)$, $F_2\left( {t}
\right)=f\left( {t} \right)$, $G_2\left( {t} \right)=g\left( {t}
\right)$ and $p\equiv p$ in \eqref{eq1.2.4}, then we obtain the
following weighted version of Pompeiu--Chebyshev functional
\begin{align*}
&\widehat{\mathcal{P}}_h\left(f,g;p\right)\\&= \int_a^b {p\left( x
\right)h^2\left( x \right)dx}  \cdot \int_a^b {p\left( x
\right)f\left( x \right)g\left( x \right)dx}
 - \int_a^b {p\left( x \right)h\left( x \right)g\left( x
\right)dx} \cdot \int_a^b {p\left( x \right)h\left( x \right)f
\left( x \right)dx}
\\
&=\frac{1}{2}\int_a^b {\int_a^b {p\left( {x_1 } \right)p\left(
{x_2 } \right)\left[ {h\left( {x_1 } \right)f\left( {x_2 } \right)
- h\left( {x_2 } \right)f\left( {x_1 } \right)} \right]\left[
{h\left( {x_1 } \right)g\left( {x_2 } \right) - h\left( {x_2 }
\right)g\left( {x_1 } \right)} \right]dx_1 }dx_2  }.
\end{align*}
\end{remark}

\begin{remark}
In his work \cite{Dragomir1}, Dragomir considered the Chebyshev
functional $\mathcal{T}\left(F_1,F_2\right)$ between two
absolutely continuous mappings $F_1(x)=\frac{f(x)}{x}$ and
$F_2(x)=\frac{g(x)}{x}$ , $x\in [a,b]$, $b>a>0$. This can be
generalized in terms of $h$-function as mentioned above (see Lemma
\ref{lemma2} above, and Theorem 2.1 in \cite{Dragomir1}). We left
this
part to the interested reader.\\
\end{remark}

\subsection{The  reverse of CBS inequality} The Pompeiu--Chebyshev functional
$\widehat{\mathcal{P}}_h \left(\cdot,\cdot\right)$
 can be very
useful to bound the reverse of  CBS inequality, which it has many
applications in various branches of Mathematics, Physics and
Statistics. In another context, we find the following related
result due to Barnett and Dragomir \cite{Barnett}:
\begin{multline}
\label{eq1.Barnett2008} \int_a^b {h^2 \left( x \right)dx} \int_a^b {f^2 \left( x \right)dx}  - \left( {\int_a^b {h\left( x \right)f\left( x \right)dx} } \right)^2  \\
  \le H^2 \left\{ \begin{array}{l}
 \frac{{\left( {b - a} \right)^{2p + 2} }}{{\left( {2p + 1} \right)\left( {2p + 2} \right)}}\left\| h \right\|_\infty ^4 ,\,\,\,\,\,\,\,\,\,\,\,\,\,\,\,\,\,\,\,\,\,\,\,\,\,{\rm{if}}\,\,\,h \in L^\infty  \left[ {a,b} \right] \\
  \\
 \frac{{2^{ - \frac{1}{\beta }} \left( {b - a} \right)^{2p + \frac{2}{\alpha }} }}{{\left( {2\alpha p + 1} \right)^{\frac{1}{\alpha }} \left( {2\alpha p + 2} \right)^{\frac{1}{\alpha }} }}\left\| h \right\|_{2\beta }^4 ,\,\,\,\,\,\,\,\,\,{\rm{if}}\,\,\,h \in L^{2\beta}   \left[ {a,b} \right] \\
  \\
 \frac{1}{2}\left( {b - a} \right)^{2p} \left\| h \right\|_2^4 , \,\,\,\,\,\,\,\,\,\,\,\,\,\,\,\,\,\,\,\,\,\,\,\,\,\,\,\,\,\,\,{\rm{if}}\,\,\,h \in L^2  \left[ {a,b} \right] \\
 \end{array} \right.,
\end{multline}
for $\alpha >1$, $\frac{1}{\alpha}+\frac{1}{\beta}=1$, where $f$
and $h$ are assumed to be measurable on $[a,b]$ and $\frac{f}{h}$
is H\"{o}lder continuous of order $p\in (0,1]$ with H\"{o}lderian
constant $H>0$, and
\begin{align*}
\left\| h \right\|_\infty   = {\rm{ess}}\mathop {\sup }\limits_{t
\in \left[ {a,b} \right]} \left| {h\left( t \right)}
\right|,\,\,\, {\rm{and}} \,\,\, \left\| h \right\|_p  = \left(
{\int_a^b {\left| {h\left( t \right)} \right|^p dt} }
\right)^{\frac{1}{p}}, p \ge 1,
\end{align*}
are the usual Lebesgue norms. Clearly, this result support our
consideration to generalize the Chebyshev functional as presented
in the functional $\widehat{\mathcal{P}}_h
\left(\cdot,\cdot\right)$. Moreover, if $p=1$ in
\eqref{eq1.Barnett2008}, i.e., $\frac{f}{h}$ satisfy the Lipschitz
condition we get
\begin{multline}
 \label{eq2.Barnett2008}\int_a^b {h^2 \left( x \right)dx} \int_a^b {f^2 \left( x \right)dx}  - \left( {\int_a^b {h\left( x \right)f\left( x \right)dx} } \right)^2  \\
  \le L^2 \left\{ \begin{array}{l}
\frac{{\left( {b - a} \right)^4 }}{{12}}\left\| h \right\|_\infty ^4 ,\,\,\,\,\,\,\,\,\,\,\,\,\,\,\,\,\,\,\,\,\,\,\,\,\,\,\,\,\,\,\, \,\,{\rm{if}}\,\,\,h \in L^\infty  \left[ {a,b} \right] \\
  \\
 \frac{{2^{ - \frac{1}{\beta }} \left( {b - a} \right)^{2 + \frac{2}{\alpha }} }}{{\left( {2\alpha  + 1} \right)^{\frac{1}{\alpha }} \left( {2\alpha  + 2} \right)^{\frac{1}{\alpha }} }}\left\| h \right\|_{2\beta }^4 ,\,\,\,\,\,\,\,\, {\rm{if}}\,\,\,h \in L^{2\beta}   \left[ {a,b} \right] \\
  \\
 \frac{1}{2}\left( {b - a} \right)^2 \left\| h \right\|_2^4 ,\,\,\,\,\,\,\,\,\,\,\,\,\,\,\,\,\,\,\,\,\,\,\,\,\,\,\,{\rm{if}}\,\,\,h \in L^2 \left[ {a,b} \right] \\
 \end{array} \right. .
\end{multline}
It was shown in \cite{Barnett}  that, the constant $\frac{1}{12}$
is the best possible. This can be seen by choosing $f(x) = x$ and
$h(x) =
1$.\\

Fortunately, $\widehat{\mathcal{P}}_h \left(f,f\right)$ represents
(exactly) the revers of CBS inequality, therefore  we can state
some related results, as follows:
\begin{theorem}
Assume that  $f$ and $g$ are measurable on $[a,b]$ and
$\frac{f}{h}$, $\frac{g}{h}$ are  H\"{o}lder continuous of order
$p,q\in (0,1]$ with H\"{o}lderian constants $H_1,H_2>0$;
respectively, then we have
\begin{align}
\label{CBS1.Alomari2017}
\left|{\widehat{\mathcal{P}}_h\left(f,g\right)}\right|  \le H_1H_2
\left\{ \begin{array}{l}
 \frac{{\left( {b - a} \right)^{p + q + 2} }}{\sqrt{\left( {2p + 1} \right) \left( {2p + 2} \right)  \left( {2q+ 1} \right) \left( {2q + 2} \right)}}\left\| h \right\|_\infty ^4 ,\,\,\,\,\,\,\,\,\,\,\,\,\,\,\,\,\,\,\,\,\,\,\,\,\,{\rm{if}}\,\,\,h \in L^\infty  \left[ {a,b} \right] \\
  \\
 \frac{{2^{ - \frac{1}{\beta }} \left( {b - a} \right)^{p +q+ \frac{2}{\alpha }} }}{{\left[\left( {2\alpha p + 1} \right)  \left( {2\alpha p + 2} \right)\left( {2\alpha q + 1} \right)  \left( {2\alpha q + 2} \right) \right]^{\frac{1}{2\alpha }} }}\left\| h \right\|_{2\beta }^4 ,\,\,\,\,\,\,\,\,\,{\rm{if}}\,\,\,h \in L^{2\beta}   \left[ {a,b} \right] \\
  \\
 \frac{1}{2}\left( {b - a} \right)^{p+q} \left\| h \right\|_2^4 , \,\,\,\,\,\,\,\qquad\qquad\,\,\,\,\,\,\,\,\,\,\,\,\,\,\,\,\,\,\,\,\,\,\,\,\,\,\,\,\,\,\,{\rm{if}}\,\,\,h \in L^2  \left[ {a,b} \right] \\
 \end{array} \right.,
\end{align}
for $\alpha >1$, $\frac{1}{\alpha}+\frac{1}{\beta}=1$. Provided
that $h(s)\ne 0$ in $[a,b]$.
\end{theorem}
\begin{proof}
From \eqref{eq1.Barnett2008}, we have
\begin{align}
\label{CBS2.Alomari2017}
\left|{\widehat{\mathcal{P}}_h\left(f,f\right)}\right| &=\int_a^b
{h^2 \left( x \right)dx} \int_a^b {f^2 \left( x \right)dx}  -
\left( {\int_a^b {h\left( x \right)f\left( x \right)dx} }
\right)^2
\nonumber\\
&\le H^2 \left\{ \begin{array}{l}
 \frac{{\left( {b - a} \right)^{2p + 2} }}{{\left( {2p + 1} \right)\left( {2p + 2} \right)}}\left\| h \right\|_\infty ^4 ,\,\,\,\,\,\,\,\,\,\,\,\,\,\,\,\,\,\,\,\,\,\,\,\,\,{\rm{if}}\,\,\,h \in L^\infty  \left[ {a,b} \right] \\
  \\
 \frac{{2^{ - \frac{1}{\beta }} \left( {b - a} \right)^{2p + \frac{2}{\alpha }} }}{{\left( {2\alpha p + 1} \right)^{\frac{1}{\alpha }} \left( {2\alpha p + 2} \right)^{\frac{1}{\alpha }} }}\left\| h \right\|_{2\beta }^4 ,\,\,\,\,\,\,\,\,\,{\rm{if}}\,\,\,h \in L^{2\beta}  \left[ {a,b} \right] \\
  \\
 \frac{1}{2}\left( {b - a} \right)^{2p} \left\| h \right\|_2^4 , \,\,\,\,\,\,\,\,\,\,\,\,\,\,\,\,\,\,\,\,\,\,\,\,\,\,\,\,\,\,\,{\rm{if}}\,\,\,h \in L^2  \left[ {a,b} \right] \\
 \end{array} \right.,
\end{align}
The same inequality holds for
$\left|{\widehat{\mathcal{P}}_h\left(g,g\right)}\right|$.
Substituting \eqref{CBS2.Alomari2017} and that one resulting from
$\left|{\widehat{\mathcal{P}}_h\left(g,g\right)}\right|$ in
\eqref{Gene.Cheb.POMP} we get the required result.
\end{proof}

\begin{corollary}
Let $b>a>0$ and assume that  $f$ and $g$ are measurable on $[a,b]$
and $\frac{f}{x}$, $\frac{g}{x}$ are  H\"{o}lder continuous of
order $p,q\in (0,1]$ with H\"{o}lderian constants $H_1,H_2>0$;
respectively, then we have
\begin{align}
\left|{\widehat{\mathcal{P}}\left(f,g\right)}\right|  \le H_1H_2
\left\{ \begin{array}{l}
 \frac{{\left( {b - a} \right)^{p + q + 2} }}{\sqrt{\left( {2p + 1} \right) \left( {2p + 2} \right)  \left( {2q+ 1} \right) \left( {2q + 2} \right)}}b^4,\\
  \\
 \frac{{2^{ - \frac{1}{\beta }} \left( {b - a} \right)^{p +q+ \frac{2}{\alpha }} }}{{\left[\left( {2\alpha p + 1} \right)  \left( {2\alpha p + 2} \right)\left( {2\alpha q + 1} \right)  \left( {2\alpha q + 2} \right) \right]^{\frac{1}{2\alpha }} }} \left( {\frac{{b^{2\beta  + 1}  - a^{2\beta  + 1} }}{{2\beta  +
1}}} \right)^{\frac{2}{\beta }},\\
  \\
 \frac{1}{2}\left( {b - a} \right)^{p+q}\left( {b^3  - a^3 } \right)^2,
 \end{array} \right.
\end{align}
for $\alpha >1$, $\frac{1}{\alpha}+\frac{1}{\beta}=1$.
\end{corollary}

\begin{proof}
Setting $h(x)=x$ in \eqref{CBS1.Alomari2017}.
\end{proof}

\begin{theorem}
Assume that  $f$ and $g$ are measurable on $[a,b]$ and
$\frac{h}{f}$, $\frac{h}{g}$ are  H\"{o}lder continuous of order
$p,q\in (0,1]$ with H\"{o}lderian constants $H_1,H_2>0$;
respectively, then we have
\begin{align}
\label{CBS4.Alomari2017}
\left|{\widehat{\mathcal{P}}_h\left(f,g\right)}\right|  \le H_1H_2
\left\{ \begin{array}{l}
 \frac{{\left( {b - a} \right)^{p + q + 2} }}{\sqrt{\left( {2p + 1} \right) \left( {2p + 2} \right)  \left( {2q+ 1} \right) \left( {2q + 2} \right)}}\left\| f \right\|_\infty ^2\left\| g \right\|_\infty ^2 ,\,\,\,\,\,\,\,\,\,\,\,\,\,\,\,\,\,\,\,\,\,\,{\rm{if}}\,\,\,f,g \in L^\infty  \left[ {a,b} \right] \\
  \\
 \frac{{2^{ - \frac{1}{\beta }} \left( {b - a} \right)^{p +q+ \frac{2}{\alpha }} }}{{\left[\left( {2\alpha p + 1} \right)  \left( {2\alpha p + 2} \right)\left( {2\alpha q + 1} \right)  \left( {2\alpha q + 2} \right) \right]^{\frac{1}{2\alpha }} }}\left\| f \right\|_{2\beta }^2\left\| g \right\|_{2\beta }^2 , \,\,\,\,\,{\rm{if}}\,\,\,f,g \in L^{2\beta}   \left[ {a,b} \right] \\
  \\
 \frac{1}{2}\left( {b - a} \right)^{p+q} \left\| f \right\|_2^2\left\| g \right\|_2^2 , \,\,\,\,\,\,\,\qquad\qquad\,\,\,\,\,\,\,\,\,\,\,\,\,\,\,\,\,\,\,\,\,\,\,\,\,\,\,\,\,\,{\rm{if}}\,\,\,f,g \in L^2  \left[ {a,b} \right] \\
 \end{array} \right.,
\end{align}
for $\alpha >1$, $\frac{1}{\alpha}+\frac{1}{\beta}=1$. Provided
that neither $f(s)$ nor $g(s)$ equal to $0$ in $[a,b]$.
\end{theorem}
\begin{proof}
From \eqref{eq1.Barnett2008}, we have
\begin{align}
\label{CBS5.Alomari2017}
\left|{\widehat{\mathcal{P}}_h\left(f,f\right)}\right| &=\int_a^b
{h^2 \left( x \right)dx} \int_a^b {f^2 \left( x \right)dx}  -
\left( {\int_a^b {h\left( x \right)f\left( x \right)dx} }
\right)^2
\nonumber\\
&\le H_1^2 \left\{ \begin{array}{l}
 \frac{{\left( {b - a} \right)^{2p + 2} }}{{\left( {2p + 1} \right)\left( {2p + 2} \right)}}\left\|f \right\|_\infty ^4 ,\,\,\,\,\,\,\,\,\,\,\,\,\,\,\,\,\,\,\,\,\,\,\,\,\,{\rm{if}}\,\,\,f \in L^\infty  \left[ {a,b} \right] \\
  \\
 \frac{{2^{ - \frac{1}{\beta }} \left( {b - a} \right)^{2p + \frac{2}{\alpha }} }}{{\left( {2\alpha p + 1} \right)^{\frac{1}{\alpha }} \left( {2\alpha p + 2} \right)^{\frac{1}{\alpha }} }}\left\| f \right\|_{2\beta }^4 ,\,\,\,\,\,\,\,\,\,{\rm{if}}\,\,\,f \in L^{2\beta}  \left[ {a,b} \right] \\
  \\
 \frac{1}{2}\left( {b - a} \right)^{2p} \left\| f \right\|_2^4 , \,\,\,\,\,\,\,\,\,\,\,\,\,\,\,\,\,\,\,\,\,\,\,\,\,\,\,\,\,\,\,{\rm{if}}\,\,\,f \in L^2  \left[ {a,b} \right] \\
 \end{array} \right.,
\end{align}
The same inequality holds for
$\left|{\widehat{\mathcal{P}}_h\left(g,g\right)}\right|$.
Substituting \eqref{CBS5.Alomari2017} and that one resulting from
$\left|{\widehat{\mathcal{P}}_h\left(g,g\right)}\right|$ in
\eqref{Gene.Cheb.POMP} we get the required result.
\end{proof}

In \cite{Barnett}, we find another inequality for pointwise
bounded functions, which reads: If  there exist constants $M \ge m
> 0$ such that $h\left(x\right)\ge0$ and $M h\left(x\right) \ge
f\left(x\right) \ge m h\left(x\right)$ for almost every (a.e.) $x
\in \left[a, b\right]$, then
\begin{multline}
\label{eq2.Barnett2008}\int_a^b {h^2 \left( x \right)dx} \int_a^b
{f^2 \left( x \right)dx}  - \left( {\int_a^b {h\left( x
\right)f\left( x \right)dx} } \right)^2
\\
\le \frac{1}{4} \cdot \frac{{\left( {M - m} \right)^2
}}{{mM}}\left( {\int_a^b {h\left( x \right)f\left( x \right)dx} }
\right)^2.
\end{multline}
A straight forward result regarding
$\left|{\widehat{\mathcal{P}}_h \left(f,g\right)}\right| $ can be
deduced as follows:
\begin{theorem}
If there exist constants $M \ge m > 0$, $N \ge n > 0$ such that
$h\left(x\right)\ge0$, $M h\left(x\right) \ge f\left(x\right) \ge
m h\left(x\right)$ and $N h\left(x\right) \ge g\left(x\right) \ge
n h\left(x\right)$ for almost every (a.e.) $x \in \left[a,
b\right]$, then
\begin{align}
\label{CBS6.Alomari2017}
\left|{\widehat{\mathcal{P}}_h\left(f,g\right)}\right| &\le
\frac{1}{4} \cdot \frac{{\left( {M - m} \right)}}{\sqrt{mM}} \cdot
\frac{{\left( {N - n} \right) }}{\sqrt{Nn}}\left| {\int_a^b
{h\left( x \right)f\left( x \right)dx} } \right| \left| {\int_a^b
{h\left( x \right)g\left( x \right)dx} } \right|.
\end{align}
\end{theorem}

\begin{proof}
From \eqref{eq2.Barnett2008}, we have
\begin{align*}
\left|{\widehat{\mathcal{P}}_h\left(f,f\right)}\right| &=\int_a^b
{h^2 \left( x \right)dx} \int_a^b {f^2 \left( x \right)dx}  -
\left( {\int_a^b {h\left( x \right)f\left( x \right)dx} }
\right)^2
\\
&\le \frac{1}{4} \cdot \frac{{\left( {M - m} \right)^2
}}{{mM}}\left( {\int_a^b {h\left( x \right)f\left( x \right)dx} }
\right)^2.
\end{align*}
Similarly, for $g$ we have
\begin{align*}
\left|{\widehat{\mathcal{P}}_h\left(g,g\right)}\right| &=\int_a^b
{h^2 \left( x \right)dx} \int_a^b {g^2 \left( x \right)dx}  -
\left( {\int_a^b {h\left( x \right)g\left( x \right)dx} }
\right)^2
\\
&\le \frac{1}{4} \cdot \frac{{\left( {N - n} \right)^2
}}{{nN}}\left( {\int_a^b {h\left( x \right)g\left( x \right)dx} }
\right)^2.
\end{align*}
Substituting both inequalities in the generalized pre-Gr\"{u}ss
inequality, i.e.,
\begin{align*}
\left|{\widehat{\mathcal{P}}_h\left(f,g\right)}\right| &\le
\left|{\widehat{\mathcal{P}}_h\left(f,f\right)}\right|^{\frac{1}{2}}
\left|{\widehat{\mathcal{P}}_h\left(g,g\right)}\right|^{\frac{1}{2}}
\\
&\le \frac{1}{4} \cdot \frac{{\left( {M - m} \right)}}{\sqrt{mM}}
\cdot \frac{{\left( {N - n} \right) }}{\sqrt{Nn}}\left| {\int_a^b
{h\left( x \right)f\left( x \right)dx} } \right| \left| {\int_a^b
{h\left( x \right)g\left( x \right)dx} } \right|,
\end{align*}
which completes the proof.
\end{proof}

\begin{remark}
$L_2$-bound for $\left|{\widehat{\mathcal{P}}_h
\left(f,g\right)}\right| $ can be obtained using the same approach
considered in the proof of Theorem \ref{thm6.Alomari2017} which is
dependent mainly on the inequality \eqref{Milo}. In this case the
obtained bound will be better than that one obtained in
\eqref{CBS1.Alomari2017}.
\end{remark}

\section{Some ramified inequalities\label{sec5}}

In this section we highlight the role of Pompeiu--Chebyshev
functional $\widehat{\mathcal{P}}_{h}\left(\cdot,\cdot\right)$ in
performing and obtaining new integral inequalities. Namely, by
employing the functional
$\widehat{\mathcal{P}}_{h}\left(\cdot,\cdot\right)$ some Hardy's
type inequalities are deduced. Another inequalities for
differentiable functions are considered.

\subsection{Hardy--Chebyshev functional}

If $f$ is nonnegative $p$-integrable $(p>1)$ function on $\left(0,
\infty\right)$, then $f$ is integrable over the interval $\left(0,
x\right)$ for each positive $x$ and
\begin{align}
\label{Hardy.inequality1}\int_0^{\infty}{\left(\frac{1}{x}\int_0^x{f\left(t\right)dt}\right)^pdx}
\le
\left({\frac{p}{p-1}}\right)^p\int_0^{\infty}{f^p\left(x\right)dx}.
\end{align}
The inequality \eqref{Hardy.inequality1} is known in literature as
\emph{Hardy Integral inequality}, which was proved by Hardy in
\cite{Hardy2}. A simple and elegant  proof that is closely to
Hardy original ideas and  appealing to P\'{o}lya's simplification
that avoids  technical details can be found in \cite{Kufner1}.

Another inequality due to Hardy \cite{Hardy1}, is that
\begin{align}
\label{Hardy.inequality2}\int_a^x {\left( {\frac{{F\left( t
\right)}}{t}} \right)^2 dt} &\le 2\int_a^x {\frac{1}{t}F\left( t
\right)f\left( t \right)dt}
\\
&\le \int_a^x {\left( {\frac{{F\left( t \right)}}{t}} \right)^2
dt} + 4\int_x^{2x} {\left( {\frac{{F\left( t \right)}}{t}}
\right)^2 dt},\nonumber
\end{align}
where $f$ is assumed to be  nonnegative  and integrable function
on $\left(a, \infty\right)$ $(a>0)$ and $F\left( x
\right)=\int_a^x{f\left( t \right)dt}$.\\

Almost one hundred year passed from the first result of Hardy
\eqref{Hardy.inequality1}. Through the last three decades, several
applications specially in differential inequalities which play a
main role  in the  theory of ordinary and partial differential
equations have been implemented and investigated. For
improvements, generalizations, extensions and useful applications
of Hardy's inequality \eqref{Hardy.inequality1} the reader may
refer to \cite{Hardy3}, \cite{Kufner2}, \cite{Kufner3} and
\cite{Opic}.\\

Next, we use Pompeiu--Chebyshev functional
$\widehat{\mathcal{P}}_{h}\left(\cdot,\cdot\right)$ to study some
inequalities of Hardy's type on bounded real interval
$\left[a,b\right]$ under some other circumstances. The approach
considered here, seems to be the first work detect or figure out
the application of Chebyshev functional and its generalizations in
studying Hardy type inequalities.\\

\noindent $\textbf{(1.)}$\,\, Let $a,b$ be  positive real numbers
with $a<b$. Let $f:\left[a,b\right]\to \mathbb{R}$ be an
absolutely continuous on $\left[a,b\right]$. In
\eqref{Pompeiu.Chebyshev.Identity}, choose
$G\left(x\right)=\frac{1}{x}$, $x\in  \left[a,b\right]$ with
$F\left(x\right)=\int_a^x{f\left(t\right)dt}$ for all $t\in
\left[a,b\right]$ and $h=1$. Assume $FG$ is integrable on
$\left[a,b\right]$ then we introduce the Hardy--Chebyshev
functional
\begin{align}
\label{Hardy.Chebyshev.ramified1} \mathcal{H} \left(F,G\right)
&=\left( {b - a} \right)\int_a^b {\frac{{F\left( t \right)}}{t}dt}
- \left[ {\ln b - \ln a} \right]\int_a^b {F\left( t \right)dt},
\end{align}
or equivalently, we write $\mathcal{H}
\left(F,G\right)=\widehat{\mathcal{P}}_{1}\left(F,G\right)$.

For instance, assume  the assumptions of Theorem
\ref{thmxx.Alomari2017} are fulfilled by our choice of $F$ and $G$
as above, then we have the \emph{Hardy type inequality }
\begin{align}
\label{Hardy.Chebyshev.ramified2} \int_a^b {\frac{{F\left( t
\right)}}{t}dt} \ge \frac{1}{\mathcal{L}\left(a,b\right)}
\cdot\int_a^b {F\left( t \right)dt},
\end{align}
where $\mathcal{L}\left(a,b\right)=\frac{b-a}{\ln b - \ln a}$ is
Logarithmic mean. The inequality \eqref{Hardy.Chebyshev.ramified2}
is sharp. Moreover, the inequality
\eqref{Hardy.Chebyshev.ramified2} is reversed if we applied
Corollary \ref{thm.cor.important} instead of Theorem
\ref{thmxx.Alomari2017}.\\

On the other hand, since $\mathcal{H}
\left(F,G\right)=\widehat{\mathcal{P}}_{1}\left(F,G\right)=\left(b-a\right)\mathcal{T}
\left(F,G\right)$, then using the bounds in
\eqref{general.gruss.inequality}, we can state the following
bounds:
\begin{proposition}\label{thm.Hardy.Chebyshev.bounds}
Let $b>a>0$. Let $f:\left[a,b\right] \to \mathbb{R}$ be an
absolutely continuous function on $\left [a,b\right]$. For
$G\left(x\right)=\frac{1}{x}$, $x\in \left[a,b\right]$ and
$F\left(x\right)=\int_a^x{f\left(t\right)dt}$, $t\in
\left[a,b\right]$, we have
\begin{align}
\left|{\mathcal{T}\left( {F,G} \right)} \right| \le\left\{
\begin{array}{l} \frac{{\left( {b - a} \right)^2 }}{12a^2}\left\|
{f} \right\|_\infty
,\,\,\,\,\,\,\,\, \qquad\qquad\qquad\qquad{\rm{if}}\,\,f \in L_{\infty}\left(\left[a,b\right]\right),\\
 \\
\frac{(b-a)^2}{4ab} \left( {M- m}
\right),\,\,\, \qquad\qquad\qquad\,\,\,\,\,\,\,\,{\rm{if}}\,\, m\le f \le M,\\
 \\
\frac{{\left( {b - a} \right)^{\frac{3}{2}}}}{{\pi ^2 }}
\left(ab\right)^{-\frac{3}{2}}\mathcal{L}_2\left(a,b\right)\cdot
\left\| {f} \right\|_2
,\,\,\,\,\,\,\,\,\,\,\,\,\,\,{\rm{if}}\,\,f,\in
L_{2}\left(\left[a,b\right]\right),\\
\\
\frac{1}{8}\left( {b - a} \right)\left( {\frac{b-a}{ab}} \right)
\left\| {f} \right\|_{\infty},\qquad\qquad\,\,\,\,\,\,\,\,\,\,
{\rm{if}}\,\, f \in L_{\infty}\left(\left[a,b\right]\right),
\end{array} \right. \label{general.Hardy.inequality}
\end{align}
where $L_s \left( {a,b} \right) = \left[ {\frac{{b^{s + 1}  - a^{s
+ 1} }}{{\left( {s + 1} \right)\left( {b - a} \right)}}}
\right]^{\frac{1}{s}}$, $s \in \mathbb{R}\backslash \left\{ {-1,0}
\right\} $ is the generalized Logarithmic mean.
\end{proposition}

On utilizing the pre-Gr\"{u}ss inequality
\eqref{gene.pre.gruss2017} (with $h\left( t \right)=1$), we have
the following result.
\begin{proposition}
Let $b>a>0$. Let $f:\left[a,b\right] \to \mathbb{R}$ be an
integrable on $\left [a,b\right]$. Then
\begin{align}
\left|{\mathcal{T}\left( {F,G} \right)} \right| \le  \left|
{\frac{1}{ab} - \left( {\frac{{\ln b - \ln a}}{{b - a}}} \right)^2
} \right|^{\frac{1}{2}} \mathcal{T}^{\frac{1}{2}} \left( {F,F}
\right),
\end{align}
where
\begin{align*}
 \mathcal{T} \left( {F,F} \right)=\frac{1}{{b - a}}\int_a^b {\left( {\int_a^x
{f\left( t \right)dt} } \right)^2 dx}  - \left( {\frac{1}{{b -
a}}\int_a^b {\left( {\int_a^x {f\left( t \right)dt} } \right)dx} }
\right)^2.
\end{align*}
\end{proposition}

The Hardy--Chebyshev  functional \eqref{Hardy.Chebyshev.ramified1}
can be extended to be of Pompeiu--Chebyshev type  as follows:
\begin{align}
\label{Hardy.Chebyshev}\mathcal{H}_h \left(F,G\right) &=\int_a^b
{h^2\left( t \right)dt}\int_a^b {\frac{{F\left( t \right)}}{t}dt}
- \int_a^b {h\left( t \right)F\left( t \right)dt} \int_a^b
{\frac{h\left( t \right)}{t}dt},
\end{align}
Clearly, $\mathcal{H}_1 \left(F,G\right)=\mathcal{H}
\left(F,G\right)$.\\

\noindent $\textbf{(2.)}$\,\, To generalize the previous
presentations and the therefore to get more general inequalities
of Hardy type, let $b>a>0$ and assume $f$ is non-negative.
Therefore,
\begin{align}
\label{Hardy.Chebyshev.ramified11} \mathcal{H}
\left(F^p,G^p\right) &=\left( {b - a} \right)\int_a^b
{\left({\frac{{F\left( t \right)}}{t}}\right)^pdt} - \frac{1}{{p -
1}} \cdot\left[
 \frac{{ab^p - ba^p }}{{a^p b^p }}
\right]\int_a^b {F^p\left( t \right)dt},
\end{align}
or equivalently, we write $\mathcal{H}
\left(F^p,G^p\right)=\widehat{\mathcal{P}}_{1}\left(F^p,G^p\right)$,
for  $p>1$.

Assume  the assumptions of Theorem \ref{thmxx.Alomari2017} are
fulfilled by our choice of $F$ and $G$ as in the previous part
\textbf{(1)}, then we have the following new inequality of Hardy's
type on bounded intervals:
\begin{align}
\label{Hardy.Chebyshev.ramified21}  \int_a^b
{\left({\frac{{F\left( t \right)}}{t}}\right)^pdt} \ge\left( {a b}
\right)^{1-p} \mathcal{L}^p_p \left( {a,b} \right)\int_a^b
{F^p\left( t \right)dt},  \qquad p>1,
\end{align}
where $\mathcal{L}_p \left( {a,b}
\right)=\left[{\frac{b^{p-1}-a^{p-1}}{\left(p - 1\right)\left(b -
a\right)}}\right]^{\frac{1}{p}}$ is the generalized Logarithmic
mean. The inequality \eqref{Hardy.Chebyshev.ramified21}  is sharp.
Moreover, the inequality is reversed if we applied Corollary
\ref{thm.cor.important} instead of Theorem
\ref{thmxx.Alomari2017}.\\

On the other hand, since $\mathcal{T}
\left(F^p,G^p\right)=\frac{1}{\left(b-a\right)^2}\widehat{\mathcal{P}}_{1}\left(F^p,G^p\right)$,
then using the bounds in \eqref{general.gruss.inequality}, we can
state the following bounds:
\begin{proposition}\label{thm.Hardy.Chebyshev.bounds1}
Let $b>a>0$. Let $f:\left[a,b\right] \to \left[0,\infty\right)$ be
a non-negative and absolutely continuous function on $\left
[a,b\right]$. For $G\left(x\right)=\frac{1}{x}$, $x\in
\left[a,b\right]$ and
$F\left(x\right)=\int_a^x{f\left(t\right)dt}$, $t\in
\left[a,b\right]$, we have
\begin{align}
&\left|{\mathcal{T}\left(F^p,G^p\right)} \right|
\nonumber\\
&\le\left\{
\begin{array}{l} \frac{{\left( {b - a} \right)^{p+1} }}{12}p^2 a^{-\left(p+1\right)}\left\|
{f} \right\|^p_\infty
,\,\,\,\,\,\,\,\,\,\,\,\,\qquad\,\,\,\,\,\,\qquad\qquad{\rm{if}}\,\,f \in L_{\infty}\left(\left[a,b\right]\right),\\
 \\
\frac{1}{4}\left( {\frac{b^p-a^p}{a^pb^p}} \right)\left( {b- a}
\right)^p\left( {M^p- m^p}
\right),\,\,\, \qquad\,\,\,\,\qquad\,\,\,\,\,\,\,\,{\rm{if}}\,\, m\le f \le M,\\
 \\
p^2\frac{{\left( {b - a} \right)^{3/2}}}{{\pi ^2 }} \left( {ab}
\right)^{ - \left( {p + \frac{1}{2}} \right)}
\mathcal{L}_{2p}^{\frac{p}{2}} \left( {a,b} \right) \cdot \left\|
{F^{p - 1} f} \right\|_2 ,\,\,\,\,\,\,{\rm{if}}\,\,f,\in
L_{2}\left(\left[a,b\right]\right),\\
\\
\frac{p}{8}\left( {b - a} \right)^{p}\left(
{\frac{b^p-a^p}{a^pb^p}} \right) \left\| {f}
\right\|^p_{\infty},\qquad\qquad\qquad\,\,\,\,\,\,\,\,\,\,\,\,\,\,\,\,\,\,
{\rm{if}}\,\, f \in L_{\infty}\left(\left[a,b\right]\right),
\end{array} \right. \label{general.Hardy.inequality1}
\end{align}

\end{proposition}

On utilizing the pre-Gr\"{u}ss inequality
\eqref{gene.pre.gruss2017} (with $h\left( t \right)=1$), we have
the following result.
\begin{proposition}
Let $b>a>0$. Let $f:\left[a,b\right] \to  \left[0,\infty\right)$
be an integrable on $\left [a,b\right]$. Then
\begin{align}
\left|{\mathcal{T}\left( {F^p,G^p} \right)} \right| \le  \left|
{\frac{{ab^{2p}  - ba^{2p} }}{{\left( {b - a} \right)\left( {2p -
1} \right)a^{2p} b^{2p} }} - \frac{{\left( {ab^p  - ba^p }
\right)^2 }}{{\left( {b - a} \right)^2 \left( {p - 1} \right)^2
a^{2p} b^{2p} }}} \right|^{\frac{1}{2}} \mathcal{T}^{\frac{1}{2}}
\left( {F^p,F^p} \right),
\end{align}
for $p>1$, where
\begin{align*}
 \mathcal{T} \left( {F^p,F^p} \right)=\frac{1}{{b - a}}\int_a^b {\left( {\int_a^x
{f\left( t \right)dt} } \right)^{2p} dx}  - \left( {\frac{1}{{b -
a}}\int_a^b {\left( {\int_a^x {f\left( t \right)dt} } \right)^pdx}
} \right)^2.
\end{align*}
\end{proposition}

A generalization of Hardy--Chebyshev--Pompeiu functional
\eqref{Hardy.Chebyshev} can be presented such as:
\begin{align*}
\mathcal{H}_h \left(F^p,G^p\right) &=\int_a^b {h^2\left( t
\right)dt}\int_a^b {\left(\frac{{F\left( t
\right)}}{t}\right)^pdt} - \int_a^b {h\left( t \right)F^p\left( t
\right)dt} \int_a^b {\frac{h\left( t \right)}{t^p}dt}, \qquad p >
1.
\end{align*}
Clearly, $\mathcal{H}_1 \left(F^p,G^p\right)=\mathcal{H}
\left(F^p,G^p\right)$.\\

\noindent $\textbf{(3.)}$\,\, Another way to deal with
Hardy--Chebyshev functional  by defining the functions
$\phi:\left[a,x\right] \to \left[0,\infty\right)$, $x>a\ge0$ and
$\Phi_a:\left[0,b\right] \to \left[0,\infty\right)$ given by
$\Phi_a \left( x \right) = \frac{1}{x-a}\int_a^x {\phi \left( t
\right)dt}$, $b> x > a\ge0$. Making use of the  Chebyshev
functional $\mathcal{T}\left(\cdot,\cdot\right)$, we can write:
\begin{align*}
\mathcal{T}\left( {\Phi_a ,\Phi_a  } \right) &=
\frac{1}{{b}}\int_0^b {\Phi_a ^{2} \left( t \right)dt}  - \left(
{\frac{1}{{b }}\int_0^b {\Phi_a  \left( t \right)dt} } \right)^2
\\
&= \frac{1}{{b}}\int_0^b {\left( {\frac{1}{x-a}\int_a^x {\phi
\left( t \right)dt} } \right)^{2} dx}  - \left( {\frac{1}{{b
}}\int_0^b {\left( {\frac{1}{x-a}\int_a^x {\phi \left( t
\right)dt} } \right) dx} } \right)^2.
\end{align*}
for $p>1$. In case $a=0$, we have \emph{the Hardy functional}:
\begin{align*}
\mathcal{T}\left( {\Phi_0 ,\Phi_0 } \right)= \frac{1}{{b}}\int_0^b
{\left( {\frac{1}{x}\int_0^x {\phi \left( t \right)dt} }
\right)^{2} dx}  - \left( {\frac{1}{{b }}\int_0^b {\left(
{\frac{1}{x}\int_0^x {\phi \left( t \right)dt} } \right) dx} }
\right)^2,
\end{align*}
and we have the following bound:
\begin{theorem}
Let $\Phi_0 : \left[0,b\right]\to \left(0,\infty\right)$ be such
that $\Phi_0 $  is convex on $\left[0,b\right]$. If $\Phi_0
^{\prime}\in L_p\left[0,b\right]$ $1< p \le q$, then the
inequality
\begin{align}
\label{eq2.2.46b}\left| {\mathcal{T}\left( {\Phi_0 ,\Phi_0 }
\right)} \right| \le b \frac{{pq\sin \left( {{\textstyle{\pi \over
p}}} \right)\sin \left( {{\textstyle{\pi \over q}}}
\right)}}{{4\pi^2 \left( {p - 1} \right)^{\frac{1}{p}} \left( {q -
1} \right)^{\frac{1}{q}} }} \left\| \Phi_0 ^{\prime}
\right\|_p\left\| \Phi_0 ^{\prime} \right\|_q,
\end{align}
holds, for all $p,q>1$ with $\frac{1}{p}+\frac{1}{q}=1$.
\end{theorem}

\begin{proof}
Utilizing the triangle inequality on the right hand side of the
identity
\begin{align}
\label{eq2.2.47b}\mathcal{T}\left( {\Phi_0,\Phi_0} \right):=
\frac{1}{{b}}\int_0^b {\left[ {\Phi_0 \left( t \right) - \Phi_0
\left({\frac{ b}{2}}\right)} \right]\left[ {\Phi_0 \left( t
\right) - \frac{1}{{b}}\int_0^b {\Phi_0 \left( s \right)ds} }
\right]dt},
\end{align}
and using the H\"{o}lder's inequality, we get
\begin{align}
&\left| {\mathcal{T}\left( {\Phi_0,\Phi_0} \right)} \right|
\nonumber\\
&= \left| {\frac{1}{{b  }}\int_0^b {\left[ {\Phi_0\left( t \right)
- \Phi_0\left( {\frac{{ b}}{2}} \right) } \right]\left[
{\Phi_0\left( t \right) - \frac{1}{{b}}\int_0^b {\Phi_0\left( s
\right)ds} } \right]dt} } \right|
\nonumber\\
&\le \frac{1}{{b}}\int_0^b {\left| {\Phi_0\left( t \right) -
\Phi_0\left( {\frac{{ b}}{2}} \right) } \right|\left|
{\Phi_0\left( t \right) - \frac{1}{{b  }}\int_0^b {\Phi_0\left( s
\right)ds} } \right|dt}
\nonumber\\
&\le \frac{1}{{b }}\left( {\int_0^b {\left| {\Phi_0\left( t
\right) - \Phi_0\left( {\frac{{ b}}{2}} \right) } \right|^p dt} }
\right)^{1/p} \times\left( {\int_0^b {\left| {\Phi_0\left( t
\right) - \frac{1}{{b }}\int_a^b {\Phi_0\left( s \right)ds} }
\right|^q dt} } \right)^{1/q}.\label{eq2.2.48b}
\end{align}
Since $\Phi_0$ is convex on $[a,b]$, then by Hermite-Hadamard
inequality; i.e.,
\begin{align*}
\Phi_0\left( {\frac{{  b}}{2}} \right) \le \frac{1}{{b }}\int_a^b
{\Phi_0\left( t \right)dt}  \le \frac{{\Phi_0\left(0 \right) +
\Phi_0\left( b \right)}}{2},
\end{align*}
which follows that
\begin{align*}
\left| {\Phi_0\left( x \right) - \frac{1}{{b }}\int_0^b
{\Phi_0\left( t \right)dt} } \right|\le \left| {\Phi_0\left( x
\right) - \Phi_0\left( {\frac{{ b}}{2}} \right)  } \right|.
\end{align*}
Now, Alomari in \cite{Alomari} proved that if $f$ is absolutely
continuous functions whose first derivative $f^{\prime}$ is
positive and $ \int_a^b {\left( {f^{\prime}\left( t \right)}
\right)^p dt} < \infty $, then for any $\xi \in (a,b)$, the
inequality
\begin{align}
\label{eq.y1}\int_a^b {\left| {f\left( t \right) - f\left( \xi
\right)} \right|^p dt} \le \left( {\frac{{p^p \sin ^p \left(
{{\textstyle{\pi  \over p}}} \right)}}{{\pi ^p \left( {p - 1}
\right)}}} \right)\left[ {\frac{{b - a}}{2} + \left| {\xi  -
\frac{{a + b}}{2}} \right|} \right]^{p} \cdot\int_a^b {\left(
{f^{\prime} \left( x \right)} \right)^p dx},
\end{align}
holds for all $p>1$.

Since $\Phi_0$ is convex and therefore is absolutely continuous on
$[0,b]$ with $\Phi_0^{\prime} \in L_p[0,b]$ $(1< p <q)$, then by
\eqref{eq.y1} we can obtain the following two inequalities which
are of great interest and not less important than the main result
 \eqref{eq2.2.46b} itself:
\begin{align}
\label{eq2.2.49b}  \left( {\int_0^b {\left| {\Phi_0\left( x
\right) - \Phi_0\left( {\frac{{ b}}{2}} \right)} \right|^p dx} }
\right)^{\frac{1}{p}}  \le \frac{{p\sin \left( {{\textstyle{\pi
\over p}}} \right)}}{{2\pi \left( {p - 1} \right)^{\frac{1}{p}}
}}b\left\| \Phi_0^{\prime} \right\|_p,
\end{align}
and
\begin{align}
\label{eq2.2.50b}\left(\int_0^b{\left| {\Phi_0\left( x \right) -
\frac{1}{{b  }}\int_0^b {\Phi_0\left( t \right)dt} } \right|^q
dx}\right)^{1/q}&\le \left( {\int_0^b {\left| {\Phi_0\left( x
\right) - \Phi_0\left( {\frac{{ b}}{2}} \right)} \right|^q dx} }
\right)^{\frac{1}{q}}
\nonumber\\
&\le \frac{{q\sin \left( {{\textstyle{\pi \over q}}}
\right)}}{{2\pi \left( {q - 1} \right)^{\frac{1}{q}} }}b\left\|
\Phi_0^{\prime} \right\|_q,
\end{align}
where, $p>1$ and $\frac{1}{p} + \frac{1}{q} = 1$. Substituting the
inequalities (\ref{eq2.2.49b}), (\ref{eq2.2.50b}) and
(\ref{eq2.2.48b}), we get the required result \eqref{eq2.2.46b}.
\end{proof}

\begin{remark}
The corresponding version of Hardy's inequality on bounded
interval $\left[0,b\right]$ $(b>0)$, is given in the inequality
\eqref{eq2.2.49b} \emph{(or \eqref{eq2.2.50b})}. To treat this
formally, we can say that: For the absolutely continuous function
$\Phi_0 : \left[0,b\right]\to \left(0,\infty\right)$ such that
$\Phi_0 ^{\prime}\in L_p\left[0,b\right]$ $(p>1)$. If
$\Phi_0\left( \frac{b}{2} \right)=0 $ \emph{(or
$\int_a^b{\phi\left(t\right)dt}=0$)}, then we the inequality
\begin{align*}
 \int_0^b {\left| {\Phi_0\left( x \right)} \right|^p dx}=
 \int_0^b {\left( {\frac{1}{x}\int_0^x {\phi \left( t \right)dt}} \right)^p dx}
\le  \frac{{p^p\sin^p \left( {{\textstyle{\pi \over p}}}
\right)}}{{2^p\pi^p \left( {p - 1} \right) }}b^p  \left\|
\Phi_0^{\prime} \right\|^p_p.\\
\end{align*}
In special case, let $b=2\pi$ and $p=2$, then
\begin{align*}
\int_0^{2\pi} {\left( {\frac{1}{x}\int_0^x {\phi \left( t
\right)dt}} \right)^2 dx} \le  4   \left\| \Phi_0^{\prime}
\right\|^2_2, \qquad 0 \le x \le 2\pi.
\end{align*}
\end{remark}

\subsection{Other inequalities}
In \eqref{Pompeiu.Chebyshev.Identity}, choose $g=\frac{1}{f}$,
with $f\left(x\right)>0$, for all $x\in \left[a,b\right]$,  $f
\left( b \right) \ne f \left( a \right)$ and $h=f^{\prime}$ (where
$f$ is assumed to be differentiable in this case), then the
following identity can be ramified from
\eqref{Pompeiu.Chebyshev.Identity}
\begin{align}
\label{Pompeiu.Chebyshev.ramified1}\widehat{\mathcal{P}}_{f^{\prime}}\left(f,\frac{1}{f}\right)
&=\left(b-a\right)\int_a^b {\left(f^{\prime} \left( x
\right)\right)^2dx} -\frac{{f^2 \left( b \right) - f^2 \left( a
\right)}}{2} \cdot \left[ {\ln f\left( b \right) - \ln f\left( a
\right)} \right]
\end{align}
Thus several inequalities can be obtained for this functional.

For instance, assume  the assumptions of Theorem
\ref{thmxx.Alomari2017} are fulfilled by our choice of $g$ and $h$
as above, then we have the new inequality
\begin{align}
\label{Pompeiu.Chebyshev.ramified2} \int_a^b {\left(f^{\prime}
\left( x \right)\right)^2dx} \ge\frac{{A\left( {f\left( a
\right),f\left( b \right)} \right)}}{{L\left( {f\left( a
\right),f\left( b \right)} \right)}} \cdot \left[ {f\left( b
\right) - f\left( a \right)} \right],
\end{align}
where $A\left(\cdot,\cdot\right)$ is the arithmetic mean and
$L\left(\cdot,\cdot\right)$ is the Logarithmic mean. The
inequality is reversed if we applied Corollary
\ref{thm.cor.important} instead of Theorem
\ref{thmxx.Alomari2017}.\\

Also, since $f \left( b \right) \ne f \left( a \right)$, then we
note that
\begin{align}
\label{Pompeiu.Chebyshev.ramified3}\widehat{\mathcal{P}}_{f^{\prime}}\left(f,f\right)
&=   \left[ {f^2 \left( b \right) - f^2 \left( a \right)} \right]
\int_a^b {\left(f^{\prime} \left( x \right)\right)^2dx} -
\frac{1}{4} \left[ {f^2 \left( b \right) - f^2 \left( a \right)}
\right]^2.
\end{align}
For instance, applying  \eqref{xx.Alomari2017} for
\eqref{Pompeiu.Chebyshev.ramified3}, we get
\begin{align*}
\int_a^b {\left(f^{\prime} \left( x \right)\right)^2dx}   \ge
\frac{1}{4} \left[ {f^2 \left( b \right) - f^2 \left( a \right)}
\right].
\end{align*}
The inequality is reversed if we applied Corollary
\ref{thm.cor.important} instead of  Theorem
\ref{thmxx.Alomari2017}.\\

Also, in \eqref{Pompeiu.Chebyshev.Identity}, choose
$g=f^{\prime\prime}$, with $f\left(x\right)>0$ for all $x\in
\left[a,b\right]$ and $h\equiv h$ (where $f$ is assumed to be
twice differentiable in this case), then the following identity
can be ramified from \eqref{Pompeiu.Chebyshev.Identity}
\begin{align}
\label{Pompeiu.Chebyshev.ramified1c}\widehat{\mathcal{P}}_{h}\left(f,f^{\prime\prime}\right)
&=\int_a^b { h^{2} \left( x \right)  dx}\int_a^b { f \left( x
\right) f^{\prime\prime} \left( x \right)dx} - \int_a^b {h\left( t
\right)f\left( t \right)dt} \int_a^b {h\left( t
\right)f^{\prime\prime}\left( t \right)dt}
\end{align}
Using the results in Sections \ref{sec2}, \ref{sec3} and
\ref{sec4}, several inequalities can be obtained for this
functional.

Among others, assume  the assumptions of Theorem
\ref{thmxx.Alomari2017} are fulfilled by our choice of $f$ and $g$
as above, then we have the new inequality
\begin{align}
\label{Pompeiu.Chebyshev.ramified2c} \int_a^b { h^{2} \left( x
\right)  dx}\int_a^b { f \left( x \right) f^{\prime\prime} \left(
x \right)dx} \ge \int_a^b {h\left( t \right)f\left( t \right)dt}
\int_a^b {h\left( t \right)f^{\prime\prime}\left( t \right)dt}
\end{align}
The inequality is reversed if we applied Corollary
\ref{thm.cor.important} instead of Theorem
\ref{thmxx.Alomari2017}.

Hence, by choosing any function $h$ we can state various
inequalities. For example, let
\begin{enumerate}
\item $h=1$, then we have
\begin{align*}
 \int_a^b { f \left( x \right) f^{\prime\prime}
\left( x \right)dx} \ge \frac{f^{\prime}\left( b
\right)-f^{\prime}\left( a \right)}{b-a}\int_a^b { f\left( x
\right)dx}.
\end{align*}
\item $h=t$, then we have
\begin{align*}
 \frac{b^3-a^3}{3} \int_a^b { f \left( x \right) f^{\prime\prime} \left(
x \right)dx} \ge \left\{ bf\left( b \right) - af\left( a \right) -
\left[ {f\left( b \right) - f\left( a \right)}
\right]\right\}\int_a^b {x f \left( x \right)  dx}.
\end{align*}
\item $h=\frac{1}{f}$, $f>0$, then we have
\begin{align*}
\int_a^b { \frac{dx}{f^{2} \left( x \right)}  }\int_a^b { f \left(
x \right) f^{\prime\prime} \left( x \right)dx} \ge
\left(b-a\right)\int_a^b {\frac{f^{\prime\prime}\left( x
\right)}{f\left(x \right)}dx}.
\end{align*}

\item $h=f^{\prime}$,  then we have
\begin{align*}
\int_a^b {  \left(f^{\prime} \left( x \right) \right)^2dx
}\int_a^b { f \left( x \right) f^{\prime\prime} \left( x
\right)dx} \ge \frac{1}{4}\left[ {f^2 \left( b \right) - f^2
\left( a \right)} \right]\left[ {\left( {f'\left( b \right)}
\right)^2  - \left( {f'\left( a \right)} \right)^2 } \right].
\end{align*}
\end{enumerate}
All inequalities are reversed if we applied Corollary
\ref{thm.cor.important} instead of  Theorem
\ref{thmxx.Alomari2017}.\\

\begin{remark}
The Pompeiu--Chebyshev functional
$\widehat{\mathcal{P}}_{\left(\cdot
\right)}\left(\cdot,\cdot\right)$ is very rich and fruitful to
generate family of integral inequalities involving functions and
their derivatives, as we have seen in the Section \ref{sec5}.
\end{remark}

\centerline{}

\centerline{}

\end{document}